\documentclass[10pt,oneside,reqno]{amsart}
\usepackage{textcomp}
\usepackage{amsmath}
\usepackage{mathrsfs}
\usepackage{amsthm}
\usepackage[linktocpage]{hyperref}
\usepackage{amsfonts,graphics,amsthm,amsfonts,amscd,latexsym}
\usepackage{epsfig}
\usepackage{tikz}
\usepackage{tikz-cd}
\usepackage{wasysym}
\usepackage{flafter}
\usepackage{mathtools}
\usepackage{comment}
\usepackage{stmaryrd}
\usepackage{enumitem}
\usepackage{epigraph}

\hypersetup{
colorlinks=true,
linkcolor=red,
citecolor=green,
filecolor=blue,
urlcolor=blue
}
\usepackage{tikz}
\usetikzlibrary{graphs,positioning,arrows,shapes.misc,decorations.pathmorphing}

\tikzset{
>=stealth,
every picture/.style={thick},
graphs/every graph/.style={empty nodes},
}

\tikzstyle{vertex}=[
draw,
circle,
fill=black,
inner sep=1pt,
minimum width=5pt,
]
\usepackage[position=top]{subfig}
\usepackage[alphabetic]{amsrefs}
\usepackage{amssymb}
\usepackage{color}

\calclayout

\usetikzlibrary{decorations.pathmorphing}
\tikzstyle{printersafe}=[decoration={snake,amplitude=0pt}]

\newcommand{\pp}{\mathbb{P}}

\newcommand{\cc}{\mathbb{C}}

\def\O#1.{\mathcal {O}_{#1}}
\def\pr #1.{\mathbb P^{#1}}
\def\af #1.{\mathbb A^{#1}}
\def\ses#1.#2.#3.{0\to #1\to #2\to #3 \to 0}
\def\xrar#1.{\xrightarrow{#1}}
\def\K#1.{K_{#1}}
\def\bA#1.{\mathbf{A}_{#1}}
\def\bM#1.{\mathbf{M}_{#1}}
\def\bL#1.{\mathbf{L}_{#1}}
\def\bB#1.{\mathbf{B}_{#1}}
\def\bK#1.{\mathbf{K}_{#1}}
\def\subs#1.{_{#1}}
\def\sups#1.{^{#1}}

\DeclareMathOperator{\Supp}{Supp}

\usepackage{tikz}
\usetikzlibrary{matrix,arrows,decorations.pathmorphing}

\newtheorem{theorem}{Theorem}[section]
\newtheorem{lemma}[theorem]{Lemma}

\newtheorem{corollary}[theorem]{Corollary}

\newtheorem{notation}[theorem]{Notation}
\theoremstyle{definition}
\newtheorem{definition}[theorem]{Definition}

\newtheorem{remark}[theorem]{Remark}

\theoremstyle{remark}

\numberwithin{equation}{section}

\usepackage[all]{xy}

\newenvironment{dedication}
        {\begin{quotation}\begin{center}\begin{em}}
        {\par\end{em}\end{center}\end{quotation}}

\newcounter{rownumber}[figure]
\newcounter{rownumber-irr}[figure]
\newcounter{rownumber-p1}[figure]
\begin{document}

\title{Del Pezzo surfaces with four log terminal singularities}

\author{Grigory Belousov}
\address{Bauman Moscow State Technical University, Moscow, Russia}
\email{belousov\_grigory@mail.ru}

\author{DongSeon Hwang}
\address{Center for Complex Geometry, Institute for Basic Science (IBS), Daejeon $34126$, Republic of Korea}
\email{dshwang@ibs.re.kr}
\maketitle
\begin{abstract}
We classify del Pezzo surfaces of Picard number one with four log terminal singular points.
\end{abstract}
\begin{dedication}
Dedicated to Yurii  (Gennadievich) Prokhorov \\ on the occasion of his 60th birthday.
\end{dedication}

\section{Introduction}

A \emph{log del Pezzo surface} is a projective surface X with only
log terminal singularities such that the anti-canonical divisor $-K_X$ is ample. For surfaces, log terminal singularities are exactly quotient singularities (\cite[Corollary 1.9]{K}), which are completely classified  (\cite{Br}).

Log del Pezzo surfaces naturally appear in the log minimal model program
(see, e.g., {\cite{KMM}}). The most interesting class of log del Pezzo surfaces
is the class of such surfaces of Picard number one. The systematic study of log del Pezzo surfaces of Picard number one was initiated by Miyanishi and Zhang (\cite{Zh}). Following their approach Kojima classified such surfaces with one singular point (\cite{Ko}). See also \cite{Zan} and \cite{GZ} for further results following this approach.

Log del Pezzo surfaces of Picard number one are classified (\cite{La}) following the approach of \cite{KeM}  in terms of blow-ups of rational surfaces. See \cite{PP1} and \cite{PP2} for a recent different approach for the classification.  Log del Pezzo surfaces of Picard number one have at most $4$ singular points by  (\cite{B} or \cite{Bel1}).

In this paper we classify log del Pezzo surfaces of Picard number one with $4$ singular points, using $\pp^1$-fibration structures and the minimal model program, based on Miyanishi and Zhang's approach. In particular, we can list the types of singular points of such surfaces.

To state the main theorem we introduce one notation.
Let $R_{ks}$ be the linear chain of rational curves $D_1,D_2,\ldots, D_r$ with the following collection of $[-D_1^2,-D_2^2,\ldots,-D_r^2]:$
\[ [m_1,\overbrace{2,\ldots,2}^{m_2-2}, \ldots,
m_{k-3}+1, \overbrace{2,\ldots,2}^{m_{k-2}-2},m_{k-1}+1,\overbrace{2,\ldots,2}^{m_k-1},s,m_k+1, \overbrace{2,\ldots,2}^{m_{k-1}-2}, \ldots,m_2+1,\overbrace{2,\ldots,2}^{m_1-2}] \]
if $k$ is even;
\[ [m_1,\overbrace{2,\ldots,2}^{m_2-2},  m_3+1,\overbrace{2,\ldots,2}^{m_4-2}, \ldots,m_k+1,s,
\overbrace{2,\ldots,2}^{m_k-1}, m_{k-1}+1, \overbrace{2,\ldots,2}^{m_{k-2}-2}, \ldots,
m_2+1,\overbrace{2,\ldots,2}^{m_1-2}] \]
if $k$ is odd; where $s\geq 1$, $k\geq 3$, $m_i\geq 2$.
Moreover, the collection for $R_{1s}$ is the following $$[m_1,s,\overbrace{2,\ldots,2}^{m_1-1}],$$
and that for $R_{2s}$ is the following $$[m_1,\overbrace{2,\ldots,2}^{m_2-2},s,m_2,\overbrace{2,\ldots,2}^{m_1-2}].$$

We work over the field $\cc$ of complex numbers.

\begin{theorem}
\label{MainT} Let $X$ be a log del Pezzo surface of Picard number one. Assume that $X$ has four singular points. Then the singularity type of $X$ is one of the following, and they are all realizable:
\begin{enumerate}
\item $X$ has two singular points of type $A_1$, one singular point of type $D_n$ with $n\geq 3$, where $D_3=A_3$, and one singular point whose dual graph is of the form $$\xymatrix@R=0.8em{
\ast \ar@{-}@/^/[r]\ar@{-}@/_/[r] & R_{k1}
}$$
where $\ast$ is a $(-(n+1))$-curve and $\ast$ intersects both the end components of $R_{k1}$.
\item $X$ has three singular points of type $A_2$ and one singular point whose Hirzebruch--Jung continued fraction is of the form
$$[m_1,\overbrace{2,\ldots,2}^{m_2-1},m_3+2,2,\ldots,2,m_{k-1}+2,\overbrace{2,\ldots,2}^{m_k-1},m_k+1,2,\ldots,2,m_2+2,\overbrace{2,\ldots,2}^{m_1-2}].$$
\item $X$ has two singular points of type $A_3$, one singular point of type $A_1$ and one singular point whose Hirzebruch--Jung continued fraction is of the form  $$[m_1,\overbrace{2,\ldots,2}^{m_2-1},m_3+2,2,\ldots,2,m_{k-1}+2,\overbrace{2,\ldots,2}^{m_k-1},m_k+2,2,\ldots,2,m_2+2,\overbrace{2,\ldots,2}^{m_1-2}].$$
\item $X$ has singular points of type $A_1$, $A_2$, $A_5$ and one singular point whose Hirzebruch--Jung continued fraction is of the form
$$[m_1,\overbrace{2,\ldots,2}^{m_2-1},m_3+2,2,\ldots,2,m_{k-1}+2,\overbrace{2,\ldots,2}^{m_k-1},m_k+3,2,\ldots,2,m_2+2,\overbrace{2,\ldots,2}^{m_1-2}].$$
\item $X$ has two rational double points of type $A_1$, one singular point $P$ whose dual graph  is
$$\xymatrix@R=0.8em{
&&&&&\bullet\\
\star\ar@{-}[r]\ar@{-}[r]&\bullet\ar@{-}[r]&\bullet\ar@{-}[r]&\cdots\ar@{-}[r]&\bullet\ar@{-}[r]&\bullet\ar@{-}[r]\ar@{-}[u]&\bullet
}$$
and one singular point $Q$ whose dual graph  is of the form
$$
\xymatrix@R=0.8em{
&&\bullet\\
\bullet\ar@{-}[r]\ar@{-}[r]&\bullet\ar@{-}[r]&\ast\ar@{-}[r]\ar@{-}[u]&\star
}
$$ where $\star$ denotes a $(-3)$-curve, $\bullet$ denotes a $(-2)$-curve, $\ast$ denotes a $(-(k-1))$-curve, $k$ is the number of irreducible components of the minimal resolution of $P$.
\item $X$ has two rational double points of type $A_1$ and two singular points $P$ and $Q$ whose dual graph  is of the form
$$
\xymatrix@R=0.8em{
&\bullet&&&&\bullet\\
\bullet\ar@{-}[r]&\ast\ar@{-}[r]\ar@{-}[u]&\ast\ar@{-}[r]&\cdots\ar@{-}[r]&\ast\ar@{-}[r]&\ast\ar@{-}[r]\ar@{-}[u]&\bullet
}
$$
where $\bullet$ denotes a $(-2)$-curve, the linear chain of $\ast$'s denotes $R_{k1}$.
\item The minimal resolution $\bar{X}$ of $X$  has a $\pp^1$-fibration structure $\phi:\bar{X} \rightarrow \pp^1$ such that $\phi$ has three fibers $F_1$, $F_2$, and $F_3$ together with a section which is a $(-n)$-curve, where $F_1$ is a linear chain of rational curves with self-intersection numbers $-2,-1,-2$; and
the dual graphs of $F_2$ and $F_3$ are the followings\\
\begin{tabular}{|c|c|}
\hline $F_2$ & $F_3$\\
\hline $\xymatrix@R=0.8em{\\
\star\ar@{-}[r]\ar@{-}[r]&\circ\ar@{-}[r]&\bullet\ar@{-}[r]&\bullet
}$ & $\xymatrix@R=0.8em{\\
\star\ar@{-}[r]\ar@{-}[r]&\circ\ar@{-}[r]&\bullet\ar@{-}[r]&\bullet
}\quad \xymatrix@R=0.8em{\\
\bullet\ar@{-}[r]\ar@{-}[r]&\bullet\ar@{-}[r]&\circ\ar@{-}[r]&\star
}$\\
$\xymatrix@R=0.8em{\\
\bullet\ar@{-}[r]\ar@{-}[r]&\bullet\ar@{-}[r]&\circ\ar@{-}[r]&\star
}$ & $\xymatrix@R=0.8em{\\
\star'\ar@{-}[r]\ar@{-}[r]&\circ\ar@{-}[r]&\bullet\ar@{-}[r]&\bullet\ar@{-}[r]&\bullet
}$\\
& $\xymatrix@R=0.8em{\\
\bullet\ar@{-}[r]\ar@{-}[r]&\bullet\ar@{-}[r]&\bullet\ar@{-}[r]&\circ\ar@{-}[r]&\star'
}$\\
& $\xymatrix@R=0.8em{\\
\star'\ar@{-}[r]\ar@{-}[r]&\circ\ar@{-}[r]&\bullet\ar@{-}[r]&\bullet\ar@{-}[r]&\bullet
}$ \\ & $\xymatrix@R=0.8em{\\
\bullet\ar@{-}[r]\ar@{-}[r]&\bullet\ar@{-}[r]&\bullet\ar@{-}[r]&\circ\ar@{-}[r]&\star'
}$\\
& $\xymatrix@R=0.8em{\\
\star''\ar@{-}[r]\ar@{-}[r]&\circ\ar@{-}[r]&\bullet\ar@{-}[r]&\bullet\ar@{-}[r]&\bullet\ar@{-}[r]&\bullet
}$\\
& $\xymatrix@R=0.8em{\\
\bullet\ar@{-}[r]\ar@{-}[r]&\bullet\ar@{-}[r]&\bullet\ar@{-}[r]&\bullet\ar@{-}[r]&\circ\ar@{-}[r]&\star''
}$\\
& $\xymatrix@R=0.8em{\\
\star\ar@{-}[r]\ar@{-}[r]&\bullet\ar@{-}[r]&\circ\ar@{-}[r]&\star\ar@{-}[r]&\bullet
}$\\
& $\xymatrix@R=0.8em{\\
\bullet\ar@{-}[r]&\star\ar@{-}[r]&\circ\ar@{-}[r]&\bullet\ar@{-}[r]&\star
}$\\
\hline $
\xymatrix@R=0.8em{\\
\bullet\ar@{-}[r]\ar@{-}[r]&\circ\ar@{-}[r]&\bullet
}
$ & $R_{k1}$\\
& $R_{ks}+A_{s-2}$ for $s=3,4,\ldots$.\\
\hline
\end{tabular}\\
where $\star''$ denotes a $(-5)$-curve, $\star'$ denotes a $(-4)$-curve, $\star$ denotes a $(-3)$-curve, $\bullet$ denotes a $(-2)$-curve, $\circ$ denotes a $(-1)$-curve.

\end{enumerate}
\end{theorem}


\noindent{\bfseries Acknowledgements.}
D. Hwang was supported by the National Research Foundation of Korea (NRF) grant funded by the Korea government (MSIT) (2021R1A2C1093787) and the Institute for Basic Science (IBS-R032-D1). The authors are grateful to Professor I. A. Cheltsov and Professor Y. G. Prokhorov for useful comments and for their help.

\section{Preliminaries}
\subsection{Some useful theorems in algebraic geometry}
\begin{theorem}[Hurwitz, see, e.g., {\cite{Har}}, Corollary 2.4, Ch. 4]
\label{Hurwitz} Let $\phi\colon X\rightarrow Y$ be a finite morphism of curves. Then \[2g(X)-2=n(2g(Y)-2)+\deg R,\] where $n=\deg\phi$, $g(X), g(Y)$ are genuses of curves, $R$ is ramification divisor.
\end{theorem}

\begin{theorem}[Hodge, see, e.g., {\cite{Har}}, Theorem 1.9, Remark 1.9.1, Ch. 5]
\label{Hodge} The intersection form on a surface $X$ has the signature $(1,\rho(X)-1)$, where $\rho(X)$ is the Picard number of $X$.
\end{theorem}

\begin{theorem}[see {\cite[Corollary 9.2]{KeM}}]
\label{BMY} Let $X$ be a rational surface with log terminal
singularities. If $\rho(X)=1$, then
\[
\sum_{P\in X}\frac{m_P-1}{m_P}\leq 3,\leqno{(*)}
\]
where
$m_P$ is the order of the local fundamental group $\pi_1(U_P-\{P\})$, where
$U_P$ is a sufficiently small neighborhood of $P$.
\end{theorem}

\subsection{Log del Pezzo surfaces of Picard number one}

We always use the following notation throughout the paper.
\begin{notation}
Let $X$ be a log del Pezzo surface of Picard number one.
Let $\pi\colon\bar{X}\rightarrow X$ be the minimal resolution, $D=\sum D_i$ be the reduced exceptional divisor where $D_i$ denotes each irreducible component. Put $\pi^*(K_{X})=K_{\bar{X}}+D^\sharp$, where $D^\sharp=\sum\alpha_i D_i$ for some $0<\alpha_i<1$.
\end{notation}

Since del Pezzo surfaces with rational double points are completely classified (see, e.g., {\cite{Fur}}, {\cite{Ma}}, {\cite{Ye}}), we may assume that $X$ has at least one singular point that is not a rational double point. Moreover, in the subsequent sections, $X$ has exactly $4$ singular points.

In this case we have a strict inequality in Theorem \ref{BMY}.

\begin{theorem}[{\cite[Theorem 1.1]{Hw}}]\label{dpBMY}
Let $X$ be a log del Pezzo surface of Picard number one. Then
\[
\sum_{P\in X}\frac{m_P-1}{m_P} < 3,
\]
where
$m_P$ is the order of the local fundamental group of the singular point $P$.
\end{theorem}

We first review some results on $\pp^1$-fibrations.

\begin{lemma}[{\cite[Lemma 1.5]{Zh}}]
\label{Zhan0}
Assume that $\phi\colon\bar{X}\rightarrow\pp^1$ is a $\pp^1$-fibration. Then the following
assertions hold:
\begin{enumerate}
\item $\#\{\text{irreducible components of $D$ not in any fiber of $\phi$}\}=1+\sum\limits_F(\#\{\text{$(-1)$-curves in $F$}\}-1)$, where $F$ moves over all singular fibers of $\phi$.
\item If a singular fiber $F$ consists only of $(-1)$-curves and $(-2)$-curves then $F$ has
one of the following dual graphs:
$$
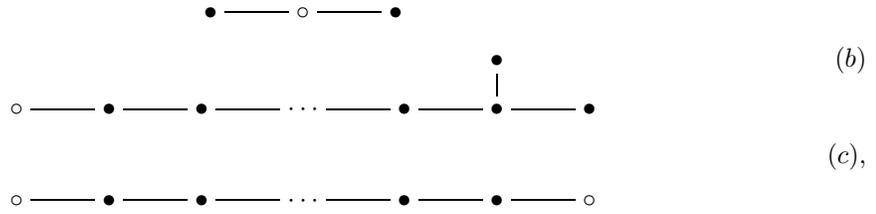
\xymatrix@R=0.8em{\\
\bullet\ar@{-}[r]\ar@{-}[r]&\circ\ar@{-}[r]&\bullet
}\eqno(a)
$$

$$
\xymatrix@R=0.8em{
&&&&&\bullet\\
\circ\ar@{-}[r]\ar@{-}[r]&\bullet\ar@{-}[r]&\bullet\ar@{-}[r]&\cdots\ar@{-}[r]&\bullet\ar@{-}[r]&\bullet\ar@{-}[r]\ar@{-}[u]&\bullet
}\eqno(b)
$$

$$
\xymatrix@R=0.8em{
\\
\circ\ar@{-}[r]\ar@{-}[r]&\bullet\ar@{-}[r]&\bullet\ar@{-}[r]&\cdots\ar@{-}[r]&\bullet\ar@{-}[r]&\bullet\ar@{-}[r]&\circ
}\eqno(c),
$$ where $\circ$ denotes a $(-1)$-curve, $\bullet$ denotes a $(-2)$-curve.
\end{enumerate}
\end{lemma}

\begin{lemma}
\label{Pr1}
Assume that there exists a $\pp^1$-fibration $\phi\colon\bar{X}\rightarrow\pp^1$ such that $C$ lies in singular fiber that contains only $(-2)$-curves and one $(-1)$-curve $C$. Then every singular fiber of $\phi$ contains only $(-1)$- and $(-2)$-curves. Moreover, every $(-1)$-curve in singular fiber is minimal.
\end{lemma}

\begin{proof}
Let $F=2C+\Delta$ be the singular fiber of $\phi$ that contain $C$. Note that $\Supp(\Delta)\subset\Supp(D)$. Let $F'$ be a singular fiber of $\phi$. Then $F'=\sum n_i E_i+\Delta'$, where $E_i$ are $(-1)$-curves, $\Supp(\Delta')\subset\Supp(D)$. Since $F\sim F'$, we see that \[2a=-(2C+\Delta)\cdot(K_{\bar{X}}+D^{\sharp})=-(\sum n_i E_i+\Delta')\cdot(K_{\bar{X}}+D^{\sharp})=\sum n_i e_i\geq(\sum n_i)a,\] where $a=-C\cdot (K_{\bar{X}}+D^{\sharp})\leq -E_i\cdot (K_{\bar{X}}+D^{\sharp})=e_i$. Then $\sum n_i=2$. Hence, every singular fiber of $\phi$ is of type as in Lemma \ref{Zhan0}. Moreover, every $(-1)$-curve in a singular fiber of $\phi$ is minimal.
\end{proof}

The following notion of minimal curves plays an important role in the paper.

\begin{definition}
A curve on $\bar{X}$ is said to be \emph{minimal} if $-C\cdot (K_{\bar{X}}+D^\sharp)$ attains the smallest positive value.
\end{definition}

\begin{lemma}[{\cite[Lemma 2.1]{Zh}}]
\label{Zhan1}
Assume that $|C+D+K_{\bar{X}}|\neq\emptyset$. Then there exists a unique decomposition $D=D'+D''$ such that $C+D''+K_{\bar{X}}\sim 0$ and $C\cdot D_i=D''\cdot D_i=K_{\bar{X}}\cdot D_i=0$ for every irreducible component $D_i$ of $D'$. In particular, $C\cdot D = 2$.
\end{lemma}

\begin{lemma}[{\cite[Lemma 2.2]{Zh}, \cite[Lemma 4.1 and Proof of Theorem 1.2]{B}}]
\label{Zhan2}
Assume that $|C+D+K_{\bar{X}}|=\emptyset$. Then we have the following.
    \begin{enumerate}
    \item $C$ is a $(-1)$-curve.
    \item $C$ intersects each connected component of $D$ at most once.
    \item If $-D_1^2\leq -D_2^2\leq\cdots\leq -D_m^2$, then $(-D_1^2,-D_2^2,\ldots, -D_m^2)$ is equal to one of the following: $(2,\ldots,2,3,5)$, $(2,\ldots,2,3,4)$, $(2,\ldots,2,3,3)$, $(2,\ldots,2,2,k)$ where $k \geq 2$ is an integer.
    \end{enumerate}
\end{lemma}

\begin{lemma}[{\cite[Lemma 2.3]{Zh}}]
\label{Zhan3}
Assume that $C$ meets at least three components $D_1,D_2,D_3$ of $D$. Put $G=2C+D_1+D_2+D_3+K_{\bar{X}}$. Then $G\sim 0$ or $G\sim\Gamma$, where $\Gamma$ is a $(-1)$-curve.
\end{lemma}

\begin{lemma}
\label{Pr2}
The curve $C$ passes through at most three irreducible components of $D$.
\end{lemma}

\begin{proof}
Assume that $C$ passes through four irreducible components $D_1,D_2,D_3,D_4$ of $D$. By Lemma \ref{Zhan1} we see that $D_1,D_2,D_3,D_4$ lie in different connected components of $D$. Put $G_1=2C+D_2+D_3+D_4+K_{\bar{X}}$, $G_2=2C+D_1+D_3+D_4+K_{\bar{X}}$, $G_3=2C+D_1+D_2+D_4+K_{\bar{X}}$, $G_4=2C+D_1+D_2+D_3+K_{\bar{X}}$. Since $G_i\cdot D_i\geq 2$ for $i=1,2,3,4$, we see that $G_i\sim\Gamma_i$, where $\Gamma_1,\Gamma_2,\Gamma_3,\Gamma_4$ are $(-1)$-curves (see Lemma \ref{Zhan3}). Note that \[\Gamma_1\cdot D_1=(2C+D_2+D_3+D_4+K_{\bar{X}})\cdot D_1=-D_1^2.\]
Put $\alpha_1$ is the coefficient of $D_1$ in $D^{\sharp}$. Assume that $D_1^2\leq -3$. Then $\alpha_1\geq\frac{1}{3}$. So, $\Gamma_1\cdot(K+D^\sharp)\geq 0$, a contradiction. Hence, $D_1^2=-2$. Then every $D_i^2=-2$ for $i=1,2,3,4$. Note that $G_i\cdot D_k\geq D_k\cdot K_{\tilde{X}}=-D_k^2-2$, where $k\neq 1,2,3,4$. As above, since $\Gamma_i\cdot(K+D^\sharp)< 0$, we see that $D_k^2\geq -3$ and $G_i$ meets every $D_k$ with $D_k^2=-3$. Note that we may assume that there exists a $(-3)$-curve on $\bar{X}$.
Let $\phi\colon\bar{X}\rightarrow\pp^1$ be a $\pp^1$-fibration defined by $|2C+D_1+D_2|$. Note that $\Gamma_3$ is in a fiber of $\phi$. Then $\Gamma_3$ is minimal (see Lemma \ref{Pr1}). On the other hand, $\Gamma_3\cdot D_3=2$ and $\Gamma$ meets every $(-3)$-components, a contradiction with Lemma \ref{Zhan1}.
\end{proof}

\begin{lemma}[{\cite[Chapter 5]{Pro1}}]
\label{WeekD} Let $Y$ be the surface obtained by extracting one irreducible curve from a log del Pezzo surface of Picard number one. Let $f\colon Y \rightarrow Z$ be any divisorial contraction. If $Z$ is a surface of Picard number one with log terminal singularities, then $-K_Z$ is ample, i.e., it is a del Pezzo surface.
\end{lemma}

\section{$\pp^1$-fibration structures on $\bar{X}$}
\label{P1Fib}

In this section we assume that there exists $\pp^1$-fibration $g\colon\bar{X}\rightarrow\pp^1$ such that there exists exactly one horizontal component $D_1$ of $D$. Moreover, we assume that $D_1$ is section or $2$-section and $D_1$ meets three other components of $D$.

\begin{lemma}
\label{DE3}
Assume that $D_1$ is a 2-section of $g$. Then $X$ has one the followings collection of singular points.
\begin{itemize}
\item
Two rational double points of type $A_1$, one singular point $P$ whose dual graph is of the form $$
\xymatrix@R=0.8em{
&&&&&\bullet\\
\star\ar@{-}[r]\ar@{-}[r]&\bullet\ar@{-}[r]&\bullet\ar@{-}[r]&\cdots\ar@{-}[r]&\bullet\ar@{-}[r]&\bullet\ar@{-}[r]\ar@{-}[u]&\bullet
}
$$ and one singular point $Q$ whose dual graph   is of the form
$$
\xymatrix@R=0.8em{
&&\bullet\\
\bullet\ar@{-}[r]\ar@{-}[r]&\bullet\ar@{-}[r]&\ast\ar@{-}[r]\ar@{-}[u]&\star
}\eqno(2)
$$ where $\star$ denotes a $(-3)$-curve, $\bullet$ denotes a $(-2)$-curve, $\ast$ denotes a $(-(k-1))$-curve, $k$ is the number of irreducible components of the minimal resolution of $P$.
\item Two rational double points of type $A_1$ and two singular points $P, Q$ whose dual graph is of the form
$$
\xymatrix@R=0.8em{
&\bullet&&&&\bullet\\
\bullet\ar@{-}[r]&\ast\ar@{-}[r]\ar@{-}[u]&\ast\ar@{-}[r]&\cdots\ar@{-}[r]&\ast\ar@{-}[r]&\ast\ar@{-}[r]\ar@{-}[u]&\bullet
}
$$
where $\bullet$ denotes a $(-2)$-curve, the linear chain of $\ast$'s denotes $R_{k1}$.
\end{itemize}
\end{lemma}

\begin{proof}
By \ref{Zhan0} we see that every singular fiber of $g$ contains only one $(-1)$-curve. Moreover, the multiplicity of $(-1)$-curve in fibers of $g$ is at least two, we see that every singular fiber meets $D_1$ in one or two points.


By Hurwitz formula we see that there exist at most two singular fibers that meet $D_1$ in one point. Since $D-D_1$ has six connected components and every singular fiber contains at most two connected components of $D-D_1$, we see that there exists a singular fiber $F_1$ that is of type (a) in Lemma \ref{Zhan0}. So, $F_1=2E_1+D_2+D_3$, where $E_1$ meets $D_1$ and $D_2,D_3$ are isolated components of $D$ that correspond to singular point of type $A_1$. Then there exists a singular fiber $F_2$ that contains two connected components of $D$ that meet $D_1$, and there exists a singular fiber $F_3$ that contains only one components of $D$ that meets $D_1$. Note that the dual graph of $F_2$ is one of the followings
$$
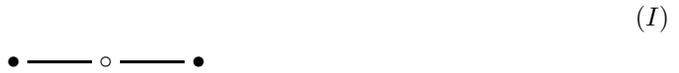
\xymatrix@R=0.8em{\\
\bullet\ar@{-}[r]\ar@{-}[r]&\circ\ar@{-}[r]&\bullet
}\eqno(I)
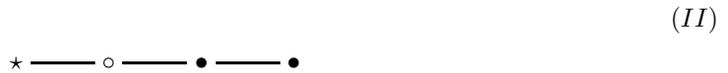
$$
$$
\xymatrix@R=0.8em{\\
\star\ar@{-}[r]\ar@{-}[r]&\circ\ar@{-}[r]&\bullet\ar@{-}[r]&\bullet
}\eqno(II)
$$ where $\star$ denotes a $(-3)$-curve, $\bullet$ denotes a $(-2)$-curve, $\circ$ denotes a $(-1)$-curve.
Let $\phi\colon\bar{X}\rightarrow Y$ be the consequence of contractions of $(-1)$-curves in $F_2$ and $F_3$. We have the followings picture on $Y$:

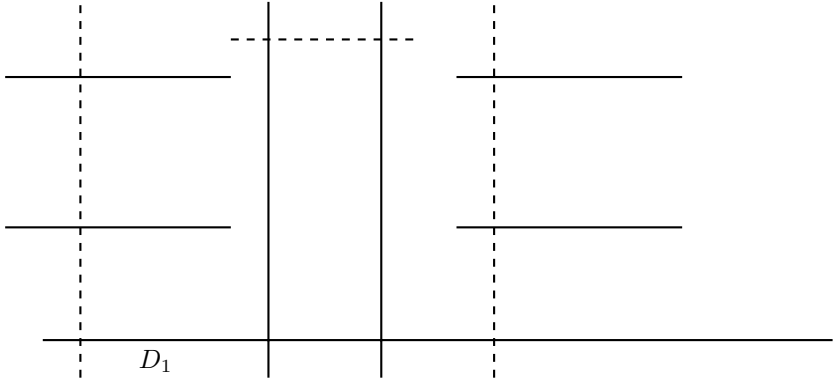
\[\begin{tikzpicture}[scale=0.5]
\draw  (0.5,0.5) -- (11,0.5) (0,2)--(3,2) (0,4)--(3,4) (3.5,0)--(3.5,5) (5,0)--(5,5) (6,2)--(9,2) (6,4)--(9,4); \draw [dashed] (1,0) -- (1,5) (3,4.5)--(5.5,4.5) (6.5,0) -- (6.5,5); \draw (2,0.5) node [below] {$D_1$};
\end{tikzpicture}
\]
where dotted lines are $(-1)$-curves solid lines except $D_1$ are $(-2)$-curves. Moreover, every dotted line and two solid lines is fiber.
Since there exists no a del Pezzo with five singular points and Picard number one, we see that $D_1^2\geq -1$ on $Y$. Note that $K_Y^2=2$. Hence, we may contract $(-1)$-curves such that we obtain $\psi\colon Y\rightarrow\pp^1\times\pp^1$ and $D_1$ is of type $(2,1)$ on $\pp^1\times\pp^1$. So, $D_1$ is a $(-1)$-curve. Then if $F_2$ has a type $(II)$, then we obtain the singular points in the first item. If $F_2$ has a type $(I)$, then we obtain the singular points in the second item.
\end{proof}

\begin{lemma}
\label{DE4}
Assume that $D_1$ is a section of $g$. Then $X$ has the singular points whose dual graph of minimal resolution is one of the followings
$$
\xymatrix@R=0.8em{
&F_1\ar@{-}[d]&\\
F_2\ar@{-}[r]\ar@{-}[r]&\ast\ar@{-}[r]& F_3
}
$$ where $\ast$ denotes a $(-n)$-curve, the dual graph of $F_1$ is the following
$$
\xymatrix@R=0.8em{\\
\bullet\ar@{-}[r]\ar@{-}[r]&\circ\ar@{-}[r]&\bullet
}
$$
the dual graphs of $F_2$ and $F_3$ are the followings
\begin{enumerate}
\item $$F_2\colon\xymatrix@R=0.8em{\\
\star\ar@{-}[r]\ar@{-}[r]&\circ\ar@{-}[r]&\bullet\ar@{-}[r]&\bullet
}\quad \xymatrix@R=0.8em{\\
\bullet\ar@{-}[r]\ar@{-}[r]&\bullet\ar@{-}[r]&\circ\ar@{-}[r]&\star
}$$
$$F_3\colon\xymatrix@R=0.8em{\\
\star\ar@{-}[r]\ar@{-}[r]&\circ\ar@{-}[r]&\bullet\ar@{-}[r]&\bullet
}\quad \xymatrix@R=0.8em{\\
\bullet\ar@{-}[r]\ar@{-}[r]&\bullet\ar@{-}[r]&\circ\ar@{-}[r]&\star
}$$
$$\xymatrix@R=0.8em{\\
\star'\ar@{-}[r]\ar@{-}[r]&\circ\ar@{-}[r]&\bullet\ar@{-}[r]&\bullet\ar@{-}[r]&\bullet
}\quad \xymatrix@R=0.8em{\\
\bullet\ar@{-}[r]\ar@{-}[r]&\bullet\ar@{-}[r]&\bullet\ar@{-}[r]&\circ\ar@{-}[r]&\star'
}$$
$$\xymatrix@R=0.8em{\\
\star''\ar@{-}[r]\ar@{-}[r]&\circ\ar@{-}[r]&\bullet\ar@{-}[r]&\bullet\ar@{-}[r]&\bullet\ar@{-}[r]&\bullet
}\quad \xymatrix@R=0.8em{\\
\bullet\ar@{-}[r]\ar@{-}[r]&\bullet\ar@{-}[r]&\bullet\ar@{-}[r]&\bullet\ar@{-}[r]&\circ\ar@{-}[r]&\star''
}$$
$$\xymatrix@R=0.8em{\\
\star\ar@{-}[r]\ar@{-}[r]&\bullet\ar@{-}[r]&\circ\ar@{-}[r]&\star\ar@{-}[r]&\bullet
}\quad \xymatrix@R=0.8em{\\
\bullet\ar@{-}[r]&\star\ar@{-}[r]&\circ\ar@{-}[r]&\bullet\ar@{-}[r]&\star
}$$ where $\star$ denotes a $(-3)$-curve, $\bullet$ denotes a $(-2)$-curve, $\circ$ denotes a $(-1)$-curve.
\item the dual graph of $F_2$ is the following  $$
\xymatrix@R=0.8em{\\
\bullet\ar@{-}[r]\ar@{-}[r]&\circ\ar@{-}[r]&\bullet
}
$$ $F_3=R_{ks}$ for $s=1,3,4,\ldots$, and $X$ has one more singular point of type $A_{s-2}$ for $s=3,4,\ldots$.
\end{enumerate}
\end{lemma}

\begin{proof}
By \ref{Zhan0} we see that every singular fiber of $g$ contains only one $(-1)$-curve. Since the multiplicity of $(-1)$-curve in fibers of $g$ is at least two, we see that every singular fiber contains connected component of $D-D_1$. By classification of log terminal singular points we obtain required classification.
\end{proof}

\section{The case $|C+D+K_{\bar{X}}|\neq\emptyset$}\label{NoEmpty}
We assume that every singular point is cyclic unless it is a rational double point. See Section \ref{noncycDV} for the case where there is a non-cyclic singular point that is not a rational double point. Throughout this section, we assume that $|C+D+K_{\bar{X}}|\neq\emptyset$ where $C$ denotes a minimal curve. So, by Lemma \ref{Zhan1}, we see that there exists a decomposition $D=D'+D''$ such that $C+D''+K_{\bar{X}}\sim 0$. We see that $C+D''$ is a wheel and $D'$ consists of $(-2)$-curves.

\begin{lemma}
\label{C(-1)curve}
Assume that $|C+D+K_{\bar{X}}|\neq\emptyset$. Then $C$ is a $(-1)$-curve.
\end{lemma}

\begin{proof}
By Lemma \ref{Zhan1} there exists a decomposition $D=D'+D''$ such that $C+D''+K_{\bar{X}}\sim 0$ and $C\cdot D_i=D''\cdot D_i=K_{\bar{X}}\cdot D_i=0$ for every irreducible component $D_i$ of $D'$. Then $C$ is a smooth rational curve and $C\cdot D=C\cdot D''=2$. Assume that $C$ meets components $D_1$ and $D_2$ of $D$ and $\alpha_1,\alpha_2$ are the coefficients of $D_1$ and $D_2$ in $D^{\sharp}$ (maybe $D_1=D_2$ and $\alpha_1=\alpha_2$). Note that $\alpha_1<1$, $\alpha_2<1$. So, $$-C\cdot(K_{\bar{X}}+D^{\sharp})=-C\cdot K_{\bar{X}}-\alpha_1-\alpha_2>-C\cdot K_{\bar{X}}-2.$$ On the other hand, let $E$ be a $(-1)$-curve. We have $-E\cdot(K_{\bar{X}}+D^{\sharp})<1$. Hence, $-C\cdot K_{\bar{X}}<3$. So, $C$ is either a $(0)$-curve, either a $(-1)$-curve. Consider a $\pp^1$-fibration $\phi\colon\bar{X}\rightarrow\pp^1$ defined by $C$. Since we may assume that $\#D>1$, we see that there exists a singular fiber $F$ of $\phi$. Put $E_1,E_2,\ldots, E_m$ are $(-1)$-curves in $F$ and $n_1,n_2,\ldots,n_k$ are multiplicity of those curves in $F$. Note that $\sum n_i\geq 2$. Since $F\cdot(K_{\bar{X}}+D^{\sharp})=C\cdot(K_{\bar{X}}+D^{\sharp})$, we see that $-E_i\cdot(K_{\bar{X}}+D^{\sharp})<-C\cdot(K_{\bar{X}}+D^{\sharp})$, a contradiction.
\end{proof}

Note that $C$ is a $(-1)$-curve. Let $P_1, P_2, P_3, P_4$ be singular points of $X$. We may assume that $P_4$ corresponds to $D''$. Let $D^{(1)},D^{(2)},D^{(3)}$ be connected components of $D$ that correspond to $P_1,P_2,P_3$ respectively. Let $\phi\colon\bar{X}\rightarrow Y$ be the consequence of contraction of $(-1)$-curves in $C+D''$. We obtain one of the followings cases\\
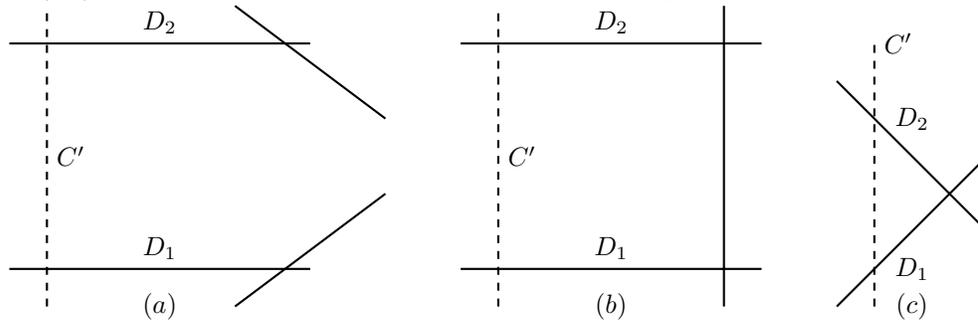
\begin{tikzpicture}[scale=0.7]
\draw  (0,0.5) -- (4,0.5) (0,3.5)--(4,3.5) (3,4)--(5,2.5) (3,0)--(5,1.5) (6,0.5) -- (10,0.5) (6,3.5)--(10,3.5) (9.5,0)--(9.5,4) (11,4)--(13,2) (11,1)--(13,3); \draw [dashed] (0.5,0) -- (0.5,4) (6.5,0) -- (6.5,4) (11.5,0.5) -- (11.5,4); \draw (2,0) node {$(a)$} (8,0) node {$(b)$} (8,0.5) node [above] {$D_1$} (8,3.5) node [above] {$D_2$} (12,1) node {$D_1$} (12,3.5) node {$D_2$} (12,0) node {$(c)$} (2,0.5) node [above] {$D_1$} (2,3.5) node [above] {$D_2$} (0.5,1.7) node [right] {$C'$} (6.5,2) node [right] {$C'$} (11.5,4.5) node [right] {$C'$};
\draw (15,1) ..controls  (18,3) .. (15,4); \draw [dashed] (16,0.5) -- (16,4); \draw (16.5,0) node {$(d)$} (17.5,2.5) node [right] {$D_3$} (16,2) node [left] {$C'$};
\end{tikzpicture}\\
where $C'$ is a $(-1)$-curve, $D_1,D_2$ are $(-2)$-curve, $D_3$ is either a $(-2)$-curve or a $(-3)$-curve. Indeed, if otherwise, $C+D$ is not SNC. By blowing up the non-SNC point sufficiently many times, and then by contracting all $(-n)$-curves with $n \geq 2$ we get a log del Pezzo surface of Picard number one with more than $4$ singular points, a contradiction to \cite{B}.

Consider the case (a). Note that there exists a $\pp^1$-fibration $Y\rightarrow\pp^1$ induces by $|2C'+D_1+D_2|$. So, there exists a $\pp^1$-fibration $g\colon\bar{X}\rightarrow\pp^1$ such that there exist exactly two components $D_3,D_4$ of $D''$ such that $D_3,D_4$ are sections of $g$, and every component of $D'$ is in fiber of $g$. Since $C$ is a $2$-section,  each of $P_1,P_2,P_3$ is of type $A_n$. By Lemma \ref{Zhan0} we see that there exists exactly one singular fiber of $g$ that has two $(-1)$-curves. On the other hand, there exists at most one singular fiber that has one $(-1)$-curve and this singular fiber has at most one of $D^{(1)},D^{(2)},D^{(3)}$. A contradiction.

Consider the case (b). Note that there exists a $\pp^1$-fibration $g\colon\bar{X}\rightarrow\pp^1$ such that there exists exactly one component $D_3$ of $D''$ such that $D_3$ is a 2-section of $g$ and every component of $D$ except $D_3$ is in fiber of $g$. Hence, as in Lemma \ref{DE3}, we see that the minimal resolution of $P_4$ has the following dual graph
$$\xymatrix@R=0.8em{
\ast \ar@{-}@/^/[r]\ar@{-}@/_/[r] & R_{k1}
}$$
where $\ast$ is a $(-n)$-curve, $P_1$, $P_2$ are of type $A_1$ and $P_3$ is of type $D_{n+1}$ ($D_3=A_3$).

Consider the case (c). Note that there exists no a $(-n)$-curves on $Y$ for $n\geq 3$. Let $\psi\colon Y\rightarrow\bar{Y}$ be the contraction of all $(-2)$-curves. We obtain a del Pezzo surface of Picard number one with at worst rational double points. Moreover, $$K_{\bar{Y}}^2=K_Y^2=(C'+D_1+D_2)^2=1.$$ By classification, we see that $Y$ has one of the followings collection of singularities $E_6+A_2$, $A_5+A_2+A_1$, $4A_2$. Since $\bar{Y}$ has four singular points, we see that the collection of singularities of $\bar{Y}$ is $4A_2$. Note that there exists a $(-1)$-curve that meets two components that correspond to one singular point (see, for example, {\cite{AN}}). Then $X$ has one singular point that the exceptional divisor is a linear chain of rational curves $D_1,D_2,\ldots, D_r$ with followings collection of $D_1^2,-D_2^2,\ldots,-D_r^2$
$$m_1,\overbrace{2,\ldots,2}^{m_2-1},m_3+2,2,\ldots,2,m_{k-1}+2,\overbrace{2,\ldots,2}^{m_k-1},m_k+1,2,\ldots,2,m_2+2,\overbrace{2,\ldots,2}^{m_1-2}$$ and three singular points of type $A_2$.

Consider the case (d) and $D_3^2=-2$. As above, let $\psi\colon Y\rightarrow\bar{Y}$ be the contraction of all $(-2)$-curves. We obtain a log del Pezzo surface of Picard number one with at worst rational double points. Moreover, $$K_{\bar{Y}}^2=K_Y^2=(C'+D_3)^2=1.$$ By classification, we see that $Y$ has one of the followings collection of singularities $E_7+A_1$, $A_7+A_1$, $A_5+A_2+A_1$, $D_6+2A_1$, $2A_3+2A_1$. Since $\bar{Y}$ has four singular points, we see that the collection of singularities of $\bar{Y}$ is $2A_3+2A_1$. Then $X$ has one singular point that the exceptional divisor is a linear chain of rational curves $D_1,D_2,\ldots, D_r$ with followings collection of $D_1^2,-D_2^2,\ldots,-D_r^2$
$$m_1,\overbrace{2,\ldots,2}^{m_2-1},m_3+2,2,\ldots,2,m_{k-1}+2,\overbrace{2,\ldots,2}^{m_k-1},m_k+2,2,\ldots,2,m_2+2,\overbrace{2,\ldots,2}^{m_1-2}$$ two singular points of type $A_3$ and one singular point of type $A_1$.

Consider the case (d) and $D_3^2=-3$. Let $h\colon Y\rightarrow Y'$ be the contraction of $C'$. Let $\psi\colon Y'\rightarrow\bar{Y}$ be the contraction of all $(-2)$-curves. We obtain a log del Pezzo surface of Picard number one with at worst rational double points. Moreover, $$K_{\bar{Y}}^2=K_{Y'}^2=1.$$ By classification, we see that $Y$ has one of the followings collection of singularities $E_7+A_1$, $A_7+A_1$, $A_5+A_2+A_1$, $D_6+2A_1$, $2A_3+2A_1$. Since $\bar{Y}$ has three singular points, we see that the collection of singularities of $\bar{Y}$ is $A_1+A_2+A_5$ or $D_6+2A_1$. The case $D_6+2A_1$ is impossible (see {\cite{Ma}, \cite{Ye} \cite{Zan}}).  Then $X$ has one singular point that the exceptional divisor is a linear chain of rational curves $D_1,D_2,\ldots, D_r$ with followings collection of $D_1^2,-D_2^2,\ldots,-D_r^2$
$$m_1,\overbrace{2,\ldots,2}^{m_2-1},m_3+2,2,\ldots,2,m_{k-1}+2,\overbrace{2,\ldots,2}^{m_k-1},m_k+3,2,\ldots,2,m_2+2,\overbrace{2,\ldots,2}^{m_1-2}$$ and the following collection of singularities $A_1+A_2+A_5$.

\section{The case $|C+D+K_{\bar{X}}|=\emptyset$}\label{Empty}
As in Section \ref{NoEmpty}, we assume that every singular point that is not a rational double point is cyclic. See Section \ref{noncycDV} for the case where there is a non-cyclic singular point that is not a rational double point. In this section we assume that $|E+D+K_{\bar{X}}|=\emptyset$ for every minimal curve $E$.
By Lemma \ref{Zhan2} (1), we may assume that $C$ is a $(-1)$-curve. By Lemma \ref{Pr2} we see that $C$ meets at most three component of $D$.

Assume that $C$ meets three component $D_1, D_2, D_3$ of $D$. We may assume that $D_1^2= -2$ (see Lemma \ref{Zhan2}). Note that we have the followings cases for $(-D_2^2,-D_3^2)$: $(2,n)$, $(3,3)$, $(3,4)$, $(3,5)$.

Assume that $D_1, D_2, D_3$ correspond to singularities $P_1,P_2,P_3$. Put $P_4$ is a fourth singular point. Let $D^{(1)}, D^{(2)}, D^{(3)}, D^{(4)}$ be the connected component of $D$ correspond to $P_1, P_2, P_3, P_4$ correspondingly. Let $m_1, m_2, m_3, m_4$ be the orders of local fundamental groups of $P_1, P_2, P_3, P_4$ correspondingly.

Since $-C\cdot(K_{\bar{X}}+D^{\sharp})>0$, we have the followings collection for $(-D_1^2,-D_2^2,-D_3^2)$: $(2,2,m)$, $(2,3,3)$, $(2,3,4)$, $(2,3,5)$. By Lemma \ref{Zhan3} we see that \[2C+D_1+D_2+D_3+K_{\bar{X}}\sim 0\text{ or }2C+D_1+D_2+D_3+K_{\bar{X}}\sim \Gamma,\] where $\Gamma$ is a $(-1)$-curve.

\subsection{The case $(2,3,5)$}
Assume that $2C+D_1+D_2+D_3+K_{\bar{X}}\sim 0$. Since $-K_{\bar{X}}\sim 2C+D_1+D_2+D_3$ and $K_{\bar{X}}^2=10-\rho(\bar{X})=9-\# D$, we see that $\#D=9-K_{\bar{X}}\cdot(2C+D_1+D_2+D_3)=11$. We obtain $\# D'=8$, where $D-D_1-D_2-D_3$. Let $\bar{X}\rightarrow Y$ be the consequence of contractions $C$ and $D_1$. Let $Y\rightarrow\bar{Y}$ be the contraction of all $(-n)$-curves ($n=2,3$). Note that $\bar{Y}$ has one rational double point of type $A_8$, $D_8$ or $E_8$, and one triple singular point. Assume that the rational double point is of type $A_8$ or $D_8$. By {\cite{Zan}} there exists a $(-1)$-curve $E$ on $\bar{X}$ such that $E\cdot D_3=1$ and $E$ meets component $D_4$ of $D'$ and $D_4$ is not an end component of $D'$. Let $W\rightarrow\bar{X}$ be the blow up of the intersection point of $D_4$ and $E$. Let $W\rightarrow W'$ be the contraction of all $(-n)$-curves ($n\geq 2$). We obtain a del Pezzo surface $W'$ with log terminal singularities and $\rho(W')=1$, a contradiction with Theorem \ref{BMY}. So, we see that $D'$ is $E_8$, which will be treated in  Lemma \ref{DE4}. Assume that $2C+D_1+D_2+D_3+K_{\bar{X}}\sim \Gamma$. Let $\alpha_2,\alpha_3$ be the coefficients of $D_2$ and $D_3$ in $D^\sharp$. Note that $\alpha_2\geq\frac{1}{3}$, $\alpha_3\geq\frac{3}{5}$. Assume that $D_2$ is not isolated component of $D$. Then $\alpha_2\geq\frac{2}{5}$. Assume that $D_3$ is not isolated component of $D$. Then $\alpha_3\geq\frac{2}{3}$. So, if $D_2$ and $D_3$ are not isolated components of $D$, then $\alpha_2+\alpha_3\geq 1$. Hence, $-C\cdot(K_{\bar{X}}+D^{\sharp})\leq 0$, a contradiction. So, $D_2$ and $D_3$ are isolated components of $D$.

Assume that $D_1$ is also an isolated component of $D$. Note that $\Gamma$ meets every component of $D$ except $D_2, D_3$ with self-intersection is less than $-2$. Note that $\Gamma\cdot D_k=-D_k^2+2$ for every irreducible component $D_k$ of $D$ except $D_2, D_3$. Note that the coefficient of $D_k$ in $D$ is at least $\frac{-D_k^2-2}{D_k^2}$. Since $\Gamma\cdot(K_{\bar{X}}+D^{\sharp})<0$, we see that $\Gamma$ meets only curves with self-intersection $-3$. Assume that $\Gamma$  meets only one component $D_4$ of $D$. Let $g\colon\bar{X}\rightarrow Y$ be the contraction of $\Gamma$. Then it is easy to see that the Picard group is generated by irreducible components of $g(D)$, and the intersection matrix of the irreducible components of $g(D)$ is negative definite, a contradiction with \ref{Hodge}. Assume that $\Gamma$  meets at least two components $D_4, D_5$ of $D$. Since $D_4, D_5$ lie in one connected component of $D$, we see that the coefficients of $D_4, D_5$ in $D^{\sharp}$ are at least $\frac{1}{2}$. Then $\Gamma\cdot(K_{\bar{X}}+D^{\sharp})\geq 0$, a contradiction. So, $D_1$ meets a component $D_4$ of $D$. Since $C\cdot(K_{\bar{X}}+D^{\sharp})<0$, we see that $D_4^2=-2$. Let $P_1, P_2, P_3, P_4$ be the singular points of $X$. We may assume that $D_1$ corresponds to $P_1$, $D_2$ corresponds to $P_2$, $D_3$ corresponds to $P_3$. Note that $\#D=12$. By theorem \ref{BMY} we see that $P_4$ is of type $A_1$.
Assume that $D_1$ meets one more component $D_5$ of $D$. We see that $D_5^2=-2$. Consider a $\pp^1$-fibration $\phi\colon\bar{X}\rightarrow\pp^1$ defined by $|2C+2D_1+D_4+D_5|$. Put $D_6$ is the component of $D$ over $P_4$. Let $F$ be the fiber of $\phi$ that contains $D_6$. By Lemma \ref{Pr1} we see that $F$ contains of $(-1)$- and $(-2)$-curves and every $(-1)$-curve in $F$ is minimal. So, $F$ is of type (a) or (c) in Lemma \ref{Zhan0}. If $F$ is of type (a), then there exists a minimal curve (a unique $(-1)$-curve in $F$) that meets four irreducible components of $D$ ($D_2,D_3$ and two curves in $F$), a contradiction with Lemma \ref{Pr2}.
If $F$ is of type (c), then there exists a minimal curve that meets three isolated irreducible components $D_2, D_3, D_6$ of $D$. We have already consider this case. So, we may assume that $D_1$ meets only one component $D_4$ of $D$. Note that $\Gamma$ meets $D_4$ and maybe one more component $D_5$ of $D$ such that $D_5^2=-3$. Let $\bar{X}\rightarrow Y$ be the contraction of $\Gamma$ and $Y\rightarrow\bar{Y}$ be the contraction of all $(-n)$-curves ($n\geq 2$). Then $\rho(\bar{Y})=1$ and $\bar{Y}$ has five log terminal singular points, a contradiction.

\subsection{The case $(2,3,4)$} Assume that $2C+D_1+D_2+D_3+K_{\bar{X}}\sim 0$. Since $-K_{\bar{X}}\sim 2C+D_1+D_2+D_3$, we see that $\# D'=7$. Let $\bar{X}\rightarrow Y$ be the consequence of contractions $C$ and $D_1$. Let $Y\rightarrow\bar{Y}$ be the contraction of all $(-2)$-curves. Note that $\rho(\bar{Y})=1$ and $\bar{Y}$ has one singular point of type $A_7$, $D_7$ or $E_7$, and one singular point of type $A_1$. By classification (see, for example, {\cite{Fur}}, {\cite{AN}}), we see that $D_7$ is impossible.

Assume that the rational double point is of type $A_7$. By {\cite{AN}} there exists a $(-1)$-curve $E$ on $\bar{X}$ such that $E\cdot D_3=1$ and $E$ meets component $D_4$ of $D'$. Let $W\rightarrow\bar{X}$ be the blow up of the intersection point of $D_4$ and $E$. Let $W\rightarrow W'$ be the contraction of all $(-n)$-curves ($n\geq 2$). We obtain a del Pezzo surface $W'$ with log terminal singularities and $\rho(W')=1$, a contradiction with \ref{BMY}. So, we see that $D'$ is $E_7$, which will be treated in Lemma \ref{DE4}.

Now, we assume that $C$ meets three component $D_1$, $D_2$, $D_3$ of $D$ with $D_1^2=-2$, $D_2^2=-3$, $D_3^2=-4$. Moreover, assume that $2C+D_1+D_2+D_3+K_{\bar{X}}\sim \Gamma$.

\begin{lemma}[{\cite{GZ}, \cite{KT}}]
\label{Case234-1}
Assume that $2C+D_1+D_2+D_3+K_{\bar{X}}\sim \Gamma$. Then
\begin{enumerate}
\item $\#D=11$;
\item $D-D_2-D_3$ contains only $(-2)$- and $(-3)$-curves;
\item at least one of $D_1$, $D_2$, $D_3$ is not isolated component of $D$;
\item $D_2$ and $D_3$ meet at most one component of $D$;
\item $D-D_2-D_3$ contains one or two $(-3)$-curves;
\item every $D^{(i)}$ contains at most one $(-n)$-curve ($n=3,4$);
\item $D_1$ does not meet a $(-3)$-curve;
\item $D_1$ meets at most one component of $D$. Moreover, $P_1$ is not a rational double point of type $D_n$ or $E_n$;
\end{enumerate}
\end{lemma}

\begin{proof}
(1) We see that $K_{\bar{X}}^2=(\Gamma-2C-D_1-D_2-D_3)\cdot K_{\bar{X}}=-2$. On the other hand, $K_{\bar{X}}^2=10-\rho(\bar{X})=9-\#D$. Then $\#D=11$.

(2) Assume that $D-D_2-D_3$ contains an irreducible component $\bar{D}$ such that $\bar{D}^2=-n\leq -4$. Note that the coefficient of $\bar{D}$ in $D^{\sharp}$ is at least $\frac{1}{2}$ and $\bar{D}\cdot\Gamma\geq n-2\geq 2$. Then $\Gamma\cdot(K_{\bar{X}}+D^{\sharp})\geq 0$, a contradiction.

(3) Assume that $D_1$, $D_2$, $D_3$ are isolated components of $D$. Assume that there exist at least two components $D_4$ and $D_5$ in $D^{(4)}$ with $D_4^2=D_5^2=-3$. Then the coefficients of $D_4$ and $D_5$ are at least $\frac{1}{2}$ and $\Gamma\cdot D_4=\Gamma\cdot D_5=1$. Hence, $\Gamma\cdot(K_{\bar{X}}+D^{\sharp})\geq 0$, a contradiction. So, there exists at most one component of $D^{(4)}$ with self-intersection $-3$. Let $g\colon\bar{X}\rightarrow Y$ be the contraction of $\Gamma$. Note that irreducible components of $g(D)$ generate Picard group. On the other hand, $g(D)$ is negative definite, a contradiction with \ref{Hodge}.

(4) Assume that $D_2$ meets two component of $D$. Then the coefficient of $D_2$ in $D^{\sharp}$ is at least $\frac{1}{2}$. Hence, $C\cdot(K_{\bar{X}}+D^{\sharp})\geq 0$, a contradiction. Assume that $D_3$ meets two component of $D$. Then the coefficient of $D_3$ in $D^{\sharp}$ is at least $\frac{2}{3}$. Hence, $C\cdot(K_{\bar{X}}+D^{\sharp})\geq 0$, a contradiction.


(5) Assume that $D-D_2-D_3$ contains no $(-3)$-curves. Let $\phi\colon\bar{X}\rightarrow Y$ be the contraction of $C$, let $\psi\colon Y\rightarrow\bar{Y}$ be the contraction of all $(-n)$-curves ($n\geq 2$). Then $\bar{Y}$ is a del Pezzo surface with log terminal singularities. Moreover, $\bar{Y}$ contains only one singular point that is not a rational double point. Let $\tilde{D}$ be the exception divisor of the minimal resolution $\psi$. Note that $\tilde{D}$ contains a unique $(-3)$-curve $\phi(D_3)$ and does not contain $(-n)$-curves for $n\geq 4$. Then $\phi(D_1)$ is a minimal curve on $Y$, and thus $|\phi(D_1)+\tilde{D}+K_Y|\neq\emptyset$. Then $D_1$, $D_2$, and $D_3$ are isolated, a contradiction to (3).

(6) Assume that one of $D^{(i)}$ contains two $(-n)$-curves ($n=3,4$). Then $\Gamma$ meets $\Gamma\cdot D^{(i)}=2$. Moreover, $\Gamma$ meets one of $(-n)$-curves and another $(-n)$-curve or curve between them. In both cases $\Gamma\cdot D^{\sharp}\geq 1$, a contradiction.

(7) Assume that $D_1$ meets $(-3)$-curve $D_4$. Since $\Gamma\cdot D_4=2$ and $\Gamma\cdot(K_{\bar{X}}+D^{\sharp})> 0$, we see that $D_4$ is a unique $(-3)$-curve in $D-D_2-D_3$. Let $\alpha_1, \alpha_2,\alpha_3,\alpha_4$ be the coefficients of $D_1, D_2,D_3,D_4$ in $D^{\sharp}$ correspondingly. We see that $\alpha_1\geq\frac{\alpha_4}{2}$.  Note that $K_{\bar{X}}^2=10-\rho(\bar{X})=9-\#D=-2$. Then \[0<(K_{\bar{X}}+D^{\sharp})^2=K_{\bar{X}}\cdot (K_{\bar{X}}+D^{\sharp})=-2+\alpha_2+2\alpha_3+\alpha_4.\] Since $C\cdot(K_{\bar{X}}+D^{\sharp})< 0$, we see that \[0>-1+\alpha_1+\alpha_2+\alpha_3\geq-1+\frac{\alpha_4}{2}+\alpha_2+\alpha_3.\] Then $-2+2\alpha_+2\alpha_3+\alpha_4>0$, a contradiction.

(8) Assume that $D_1$ meets two components $D_4$ and $D_5$ of $D$. Then $D_4^2=D_5^2=-2$. Let $\phi\colon\bar{X}\rightarrow\pp^1$ be a $\pp^1$-fibration defined by $|2C+2D_1+D_4+D_5|$. By Lemma \ref{Pr1} every singular fiber of $\phi$ consists of $(-1)$- and $(-2)$-curves. Moreover, every $(-1)$-curve in fibers of $\phi$ is minimal. Since the exists a $(-3)$-curve $D_6$, we see that $D_6$ meets $D_4$ or $D_5$. Let $\alpha_1, \alpha_2,\alpha_3$ be the coefficients of $D_1, D_2,D_3$ in $D^{\sharp}$ correspondingly. We see that $\alpha_1\geq\frac{2}{9}$, $\alpha_2\geq\frac{1}{3}$, $\alpha_3\geq\frac{1}{2}$. Then \[C\cdot(K_{\bar{X}}+D^{\sharp})=-1+\alpha_1+\alpha_2+\alpha_3>0,\] a contradiction. Assume that $P_1$ is a rational double point of type $D_n$ or $E_n$. Then there exists a $\pp^1$-fibration $\phi\colon\bar{X}\rightarrow\pp^1$ such that $C$ is a component of singular fiber of type (b) in Lemma \ref{Zhan0}. So, every singular fiber of $\phi$ consists of $(-1)$- and $(-2)$-curves. On the other hand, there exists a $(-3)$-curve in $D^{(4)}$. So, there exists a singular fiber of $\phi$ that consists of $(-3)$-curve, a contradiction.
\end{proof}

\begin{lemma}
\label{Case234-2}
Every singular point of $X$ is cyclic.
\end{lemma}

\begin{proof}
Note that $P_1, P_2, P_3$ are cyclic. Assume that $P_4$ is not cyclic. Then $m_2\geq 3$, $m_3\geq 4$, $m_4\geq 8$. If $P_4$ is not a rational double point, then $m_4 \geq 24$, which is a contradiction by Theorem \ref{dpBMY}, and Lemma \ref{Case234-1} (3). Assume that $P_4$ is a rational double point. Then, by Lemma \ref{Case234-1} (5), at least one of the connected components of $D_1$, $D_2$ and $D_3$ contains a $(-3)$-curve. By Theorem \ref{dpBMY} we see that the connected component $D^{(3)}$ of $D_3$ contains the $(-3)$-curve. Again, by Theorem \ref{dpBMY}, we see that $m_1 = 2$, $m_2 = 3$, $m_4 = 8$, and $m_3 \leq 23$. Then $P_1$ is of type $A_1$, $P_2$ is of type $\frac{1}{3}(1,1)$, and $P_4$ is of type $D_4$. By Lemma \ref{Case234-1} (1), $D^{(3)}$ consists of $5$ irreducible components, a contradiction to Theorem \ref{dpBMY}.
\end{proof}

\begin{lemma}
\label{Case234-3}
There exists exactly one $(-3)$-curves in $D-D_2-D_3$.
\end{lemma}

\begin{proof}
As above, $\# D^{(1)}\geq 3$ and $m_1\geq 7$. Assume that $(\# D^{(1)},\# D^{(2)},\# D^{(3)},\# D^{(4)})=(8,1,1,1)$. Then $m_1\geq 17$, $m_2=3$, $m_3=4$, $m_4=3$. We have a contradiction with \ref{BMY}. Assume that $\# D^{(2)}\geq 2$. Then $m_1\geq 7$, $m_2\geq 5$, $m_3\geq 4$, $m_4\geq 3$. Also, we have a contradiction with \ref{BMY}. The same for $\# D^{(2)}\geq 2$. Assume that $\# D^{(3)}\geq 2$. Then $m_1\geq 7$, $m_2\geq 3$, $m_3\geq 7$, $m_4\geq 3$. Also, we have a contradiction with \ref{BMY}.
\end{proof}

\begin{lemma}
\label{Case234-4}
There exist at least two components of $D-D_1-D_2-D_3$ that meet $D_1+D_2+D_3$.
\end{lemma}

\begin{proof}
Assume that there exists only one component $D_4$ of $D$ that meets one of $D_i$, $i=1,2,3$. Then $\Gamma$ meets $D_4$ and $(-3)$-curve. Let $\bar{X}\rightarrow Y$ be the contraction of $\Gamma$, let $Y\rightarrow\bar{Y}$ be the contraction of all $(-n)$-curves ($n\geq 2$). Then $\bar{Y}$ is a del Pezzo surface with log terminal singularities and $\rho(\bar{Y})=1$. We obtain $2C+D_1+D_2+D_3+K_{Y}\sim 0$ and $D$ has no connected components isomorphic to $E_7$, a contradiction.
\end{proof}

\begin{lemma}
\label{Case234-5}
There exist no two components $D_4$ and $D_5$ that meet $D_2$, $D_3$ correspondingly.
\end{lemma}

\begin{proof}
Assume that $D_4$ and $D_5$ be components of $D$ that meet $D_2$, $D_3$ correspondingly. Let $\alpha_2,\alpha_3$ be the coefficients of $D_2, D_3$ in $D^{\sharp}$. Assume that $D_4$ meets one more component other than $D_2$. Then $\alpha_2\geq\frac{3}{7}$ and $\alpha_3\geq\frac{4}{7}$. Assume that $D_5$ meets one more component other than $D_3$. Then $\alpha_2\geq\frac{2}{5}$ and $\alpha_3\geq\frac{3}{5}$. In both cases $C\cdot(K_{\bar{X}}+D^{\sharp})\geq 0$, a contradiction. So, $D^{(2)}=D_2+D_4$, $D^{(3)}=D_3+D_5$. Then $m_2=5$, $m_3=7$. Assume that $m_1=2$. Since $\#D=11$, we see that $m_4\geq 13$, a contradiction with \ref{BMY}. Assume that  $m_1=3$. Since $\#D=11$, we see that $m_4>6$, a contradiction with \ref{BMY}.  Assume that $m_1\geq 4$ and $m_4 \geq 4$. Also we have a contradiction with \ref{BMY}.
\end{proof}

Assume that $D_2$ is an isolated component of $D$, $D_1$ and $D_3$ are not isolated components of $D$. Let $D_4$ be a component that meets $D_1$, $D_5$ be a component that meets $D_3$.
Let $\phi\colon\bar{X}\rightarrow\pp^1$ be a $\pp^1$-fibration defined by $|3C+2D_1+D_4+D_2|$. Since $2C+D_1+D_2+D_3+K_{\bar{X}}\sim \Gamma$, we see that every $(-1)$-curve in fibers $\phi$ meets $D_3$. Indeed, let $E$ be a $(-1)$-curve in a fiber of $\phi$. Then $$0\leq E\cdot\Gamma=E\cdot(2C+D_1+D_2+D_3+K_{\bar{X}})=E\cdot D_3+E\cdot K_{\bar{X}}=E\cdot D_3-1.$$ So $E\cdot D_3\geq 1$. Let $F$ be the fiber of $\phi$ that contains $D_5$. Since $D_5$ meets $D_3$ and every $(-1)$-curve in fibers $\phi$ meets $D_3$, we see that $F$ consists of $(-1)$- and $(-2)$-curves. Indeed, if $F$ does not consist of $(-1)$- and $(-2)$-curves, then $\sum m_i\geq 3$ where $m_i$ is multiplicity of $(-1)$-curves in $F$. So, $F\cdot D_3\geq\sum m_i+1\geq 4$, a contradiction to the fact that $D_3$ is a $3$-section.  So, $F$ has one of type (a), (b), (c) in Lemma \ref{Zhan0}.
Note that $\phi$ has at most two horizontal component in $D$. Assume that $F$ is of type (b).
Then $\phi$ has only one horizontal component in $D$ and $P_1$ is a rational double point of type $A_2$. Put $E$ is the $(-1)$-curve in $F$. Then \[\frac{6}{7}=-E\cdot(K_{\bar{X}}+D^{\sharp})<-C\cdot(K_{\bar{X}}+D^{\sharp})=\frac{19}{21},\] a contradiction.
We claim that $D^{(3)}=D_3+D_5$. Indeed, if otherwise,  $D_5$ intersects a component $D_6$ other than $D_3$. Then $F$ is of type (c). So there exists a $(-1)$-curve $E$ such that $E\cdot D_5= E\cdot D_3=1$. Thus  $E\cdot(K_{\bar{X}}+D^{\sharp})\geq 0$, a contradiction.
Assume that $F$ is of type (c).  Then there exists a $(-1)$-curve $E$ such that $E$ meets $D_3$ and $D_5$ and does not meet any other components of $D$. Let $Y\rightarrow\bar{X}$ be the blow up of intersection point of $D_3$ and $E$. Let $Y\rightarrow\bar{Y}$ be the contraction of all $(-n)$-curves ($n\geq 2$). Then $\bar{Y}$ has four singular points and $\rho(\bar{Y})=1$, a contradiction with \ref{BMY}. So, $F$ is of type (a). We obtain $F=D_5+2E+D_6$, where $E$ is a $(-1)$-curve and $D_6$ is a $(-2)$-curve. Note that $D_6$ is not isolated component of $D$. Indeed, assume that $D_6$ is isolated component of $D$. Then $D_6=D^{(4)}$ and $D_3$ is a unique horizontal component of $D$. Then $D^{(1)}=D_1+D_4$. Hence, $\#D=6$, a contradiction. So, $D_6$ is a component of $D^{(1)}$. Let $F'$ be a singular fiber of $\phi$ that contains $D^{(4)}$. Since every $(-1)$-curves $E'$ in $F'$ meets $D_3$, $E'\cdot D_3=1$ and $D_3$ is a 3-section, we see that $F'$ is not of type (a), (b), (c) of Lemma \ref{Zhan0}. Then $D^{(4)}$ contains a $(-3)$-curve and $\# D^{(4)}=4$. So, $P_1$ is a rational double point of type $A_4$. We have $m_1=5, m_2=3, m_3=7, m_4\geq 9$, a contradiction with \ref{BMY}.

Assume that $D_3$ is isolated component of $D$, $D_1$ and $D_2$ are not isolated components of $D$. Let $D_4$ be a component that meets $D_1$, $D_5$ be a component that meets $D_2$. Let $\alpha_2, \alpha_3$ be the coefficients of $D_2$ and $D_3$ in $D^{\sharp}$, let $\beta$ be the coefficient of $(-3)$-curve $D'$ in $D^{\sharp}$. We have $\alpha_2<\frac{1}{2}$, $\alpha_3=\frac{1}{2}$. Since $(K_{\bar{X}}+D^{\sharp})^2>0$, we see that $\alpha_2+\beta>1$.  Indeed, $$-1=K_{\bar{X}}\cdot\Gamma=K_{\bar{X}}\cdot(2C+D_1+D_2+D_3+K_{\bar{X}})=-2+1+2+K_{\bar{X}}^2.$$ So, $K_{\bar{X}}^2=-2$. Then $$0<(K_{\bar{X}}+D^{\sharp})^2=K_{\bar{X}}\cdot(K_{\bar{X}}+D^{\sharp})=K_{\bar{X}}^2+\alpha_2+2\alpha_3+\beta=-1+\alpha_2+\beta.$$ Hence, $\beta>\frac{1}{2}$. Let $\phi\colon\bar{X}\rightarrow\pp^1$ be a $\pp^1$-fibration defined by $|C+D_1+D_4+\Gamma|$. Since every $D_2, D_3$ and $D'$ are sections of $\phi$, we see that every singular fiber consists of $(-1)$-curves and $(-2)$-curves. Moreover, since there exist at least four sections of $\phi$, we see that there exist at least three singular fibers $F_1, F_2, F_3$ of type (c). Since $D'$ meets at most two components of $D$, we see that $D'$ meets a $(-1)$-curve in a fiber $F_j$ for some $j$. We may assume that $j=1$, i.e. $D'$ intersects a $(-1)$-curve $E$ in $F_1$. Since $2C+D_1+D_2+D_3+K_{\bar{X}}\sim \Gamma$  and $C\cdot E=D_1\cdot E=\Gamma\cdot E=0$, we see that $E$ meets $D_2$ or $D_3$.
Assume that $E$ meets $D_3$. Then
\[E\cdot(K_{\bar{X}}+D^{\sharp})\geq-1+\alpha_3+\beta=\beta-\frac{1}{2}> 0,\] a contradiction. Assume that $E$ meets $D_2$. Then \[-E\cdot(K_{\bar{X}}+D^{\sharp})\leq 1-\alpha_2-\beta<\frac{1}{2}-\alpha_2=-C\cdot(K_{\bar{X}}+D^{\sharp}),\]
a contradiction.

\subsection{The case $(2,3,3)$} Assume that $2C+D_1+D_2+D_3+K_{\bar{X}}\sim 0$. Since $-K_{\bar{X}}\sim 2C+D_1+D_2+D_3$, we see that $\# D'=6$. Let $\bar{X}\rightarrow Y$ be the contraction of $C$. Let $Y\rightarrow\bar{Y}$ be the contraction of all $(-2)$-curves. Note that $\bar{Y}$ has two singular points $P_1$ and $P_2$, where $P_1$ is of type $A_2$ and $P_2$ is of type one of $A_6, D_6, E_6$. By the classification of del Pezzo surfaces with rational double points, we see  that $P_2$ is of type $E_6$, which will be treated in Lemma \ref{DE4}.

Assume that $2C+D_1+D_2+D_3+K_{\bar{X}}\sim \Gamma$.

\begin{lemma}[{\cite{GZ}, \cite{KT}}]
\label{Case233-1} The following assertions hold:
\begin{enumerate}
\item $\#D=10$;
\item $D-D_2-D_3$ contains only $(-2)$- and $(-3)$-curves;
\item at least one of $D_1$, $D_2$, $D_3$ is not isolated component of $D$;
\item $D-D_2-D_3$ contains one or two $(-3)$-curves;
\item every $D^{(i)}$ contains at most one $(-3)$-curve;
\end{enumerate}
\end{lemma}

\begin{proof}
The idea is the same as in Lemma \ref{Case234-1}.
\end{proof}

\begin{lemma}
\label{Case233-2}
The components $D_2$ and $D_3$ meet at most one component of $D$.
\end{lemma}

\begin{proof}
Assume that $D_2$ meets two components $D_4$ and $D_5$. Note that $D_4^2=D_5^2=-2$ since otherwise $\Gamma \cdot (K_{\bar{X}}+D^\#) > 0$, a contradiction. We see that $\Gamma\cdot D_4=\Gamma\cdot D_5=1$. Let $\phi\colon\bar{X}\rightarrow\pp^1$ be a $\pp^1$-fibration defined by $|2\Gamma+D_4+D_5|$. Let $F$ be a fiber that contains $C, D_1, D_3$. Then $F=2C+D_1+D_3+E$ where $E$ is a $(-1)$-curve that meets $D_3$. By Lemma \ref{Case233-1} we see that $\Gamma$ meets a $(-3)$-curve $D_6$. Then if
$D_6$ does not intersect $D_3$, $E\cdot D_6=2$; and if $D_6$ intersects $D_3$,  $E\cdot D_6 = E \cdot D_3 = 1$. In either case, $E\cdot(K_{\bar{X}}+D^{\sharp})\geq 0$, a contradiction.
\end{proof}

\begin{corollary}
\label{Case233-3}
Let $g\colon\bar{X}\rightarrow Y$ be the contraction of $C$. Then $g(D^{(2)}+D^{(3)})$ is an exceptional divisor over a singular point of type $A_m$ for some $m$.
\end{corollary}

\begin{lemma}
\label{Case233-4}
The divisor $D-D_2-D_3$ contains only one $(-3)$-curve.
\end{lemma}

\begin{proof}
Assume that $D-D_2-D_3$ contains two $(-3)$-curves. Let $g\colon\bar{X}\rightarrow Y$ be the contraction of $C$, let $Y\rightarrow\bar{Y}$ be the contraction of all $(-n)$-curves ($n=2,3$). We obtain a del Pezzo surface with Picard number one and with log terminal singularities. Since $\Gamma$ meets every $(-3)$-curves, we see that $\Gamma$ is a minimal curve. Put $\tilde{D}=\sum\tilde{D}_i$, where $\tilde{D}_i$ are all $(-2)$- and $(-3)$-curves. Then one of $(-3)$-curves is contained in $D^{(i)}$ $i=1,2,3$. Hence $\Gamma\cdot  D^{(i)}=2$. So, $|\Gamma+\tilde{D}+K_Y|\neq\emptyset$. On the other hand, $\Gamma$ passes through two different connected components of $\tilde{D}$, a contradiction. Hence the conclusion follows from Lemma  \ref{Case233-1} (4).
\end{proof}


\begin{lemma}
\label{Case233-5}
$D_1$ does not meet a $(-3)$-curve.
\end{lemma}

\begin{proof}
Assume that $D_1$ meets a $(-3)$-curve $D_4$. Then $\Gamma\cdot D_4=2$. Let $\alpha_1,\alpha_2,\alpha_3,\alpha_4$ be the coefficients of $D_1, D_2, D_3, D_4$ in $D^{\sharp}$ correspondingly. Note that $\alpha_4\geq\frac{2}{5}$. If $D_2$ meets a component $D_5$ of $D$ and $\alpha_5$ is the coefficient of $D_5$ in $D^{\sharp}$, then $\alpha_5\geq\frac{1}{5}$, so $\Gamma\cdot D^{\sharp}\geq 1$, a contradiction. Hence $D_2$ and $D_3$ are isolated components of $D$. If  $D_4$ meets another component $D_5$ of $D$, then $\alpha_4\geq\frac{1}{2}$, so $\Gamma\cdot D^{\sharp}\geq 1$, a contradiction. If $D_1$ meets another component $D_5$ of $D$ and $\alpha_5$ is the coefficient of $D_5$ in $D^{\sharp}$, then $\Gamma\cdot D_5=1$, $\alpha_4\geq\frac{3}{7}$, and $\alpha_5\geq\frac{1}{7}$. Thus $\Gamma\cdot D^{\sharp}\geq 1$, a contradiction. Hence $D^{(1)}=D_1+D_4$, $D^{(2)}=D_2$, $D^{(3)}=D_2$, and $D^{(4)}$ consists of six $(-2)$-curves. Let $g\colon\bar{X}\rightarrow Y$ be the contraction of $C$ and $\Gamma$, let $Y\rightarrow\bar{Y}$ be the contraction of all $(-2)$-curves. We obtain a del Pezzo surface of degree one  with Picard number one. Moreover, $\bar{Y}$ has two singular points,  one of them being of type $A_2$. By the classification results (see, for example, {\cite{AN}}, {\cite{Fur}}) we see that the dual graph of $D^{(4)}$ is $E_6$. Thus we have $m_1=5, m_2=m_3=3, m_4=24$. We have \[\frac{m_1-1}{m_1}+\frac{m_2-1}{m_2}+\frac{m_3-1}{m_3}+\frac{m_4-1}{m_4}\geq 3,\] a contradiction (see Theorem \ref{BMY}).
\end{proof}

\begin{lemma}
\label{Case233-6}
$D_1$ meets at most one component of $D$.
\end{lemma}

\begin{proof}
Assume that $D_1$ meets two components $D_4, D_5$ of $D$. Note that $D_4^2=D_5^2=-2$. Let $\phi\colon\bar{X}\rightarrow\pp^1$ be a $\pp^1$-fibration defined by $|2C+2D_1+D_4+D_5|$. By Lemma \ref{Pr1} we see that every singular fiber of $\phi$ consists of $(-1)$- and $(-2)$-curves. So, we may assume that $D_4$ meets $(-3)$-curve $D_6$. Let $g\colon\bar{X}\rightarrow Y$ be the contraction of $C$, let $Y\rightarrow\bar{Y}$ be the contraction of all $(-n)$-curves ($n=2,3$). We obtain a del Pezzo surface with Picard number one and with log terminal singularities. Since $\Gamma$ meets the $(-3)$-curve $D_6$ and meets $D_4$, we see that $\Gamma$ is minimal and $|\Gamma+\tilde{D}+K_Y|\neq\emptyset$, where $\tilde{D}$ is exceptional divisor of minimal resolution. On the other hand, $\Gamma$ meets two connected components of $\tilde{D}$, a contradiction.
\end{proof}

\begin{lemma}
\label{Case233-7}
The singular point $P_1$ is not a rational double point of type $D_n$ or $E_n$.
\end{lemma}

\begin{proof}
Assume that $P_1$ is a rational double point of type $D_n$ or $E_n$. By Lemma \ref{Case233-6}, $D_1$ intersects an irreducible component $D_4$ of $D$.
Note that there exist $(-2)$-curves
$D_4,D_5,\ldots, D_k$ such that $D_i$ meets $D_{i+1}$ and $D_k$ is the central component, i.e., it intersects three other irreducible components of $D$. Let $D_{k+1}$ and $D_{k+2}$ be the other two irreducible components of $D$ that intersect $D_k$. Hence, $|2C+2D_1+2D_4+\cdots+2D_k+D_{k+1}+D_{k+2}|$ defines a $\pp^1$-fibration $\phi\colon\bar{X}\rightarrow\pp^1$ and the  singular fiber that contains $C$ is of type (b) of Lemma \ref{Zhan0}. By Lemma \ref{Pr1}, every singular fiber of $\phi$ contains only $(-1)$-curves and $(-2)$-curves. Since $D^{(4)}$ is contained in a fiber, it consists of $(-2)$-curves, a contradiction to Lemma \ref{Case233-1} (4).
\end{proof}

\begin{lemma}
\label{Case233-8}
The divisor $D_1+D_2+D_3$ meets at least two irreducible components of $D-D_1-D_2-D_3$.
\end{lemma}

\begin{proof}
Assume that $D_1+D_2+D_3$ meets only one irreducible component $D_4$ of $D-D_1-D_2-D_3$. Then $D^(4)$ has $6$ irreducible components.  By Lemma \ref{Case233-1} and Lemma  \ref{Case233-4}, we see that $\Gamma$ meets $D_4$ and a $(-3)$-curve $D_5$. By blowing up the intersection points of $\Gamma$ and $D_5$ sufficiently many times, we derive a contradiction to Lemma \ref{BMY}.
\end{proof}

\begin{lemma}
\label{Case233-9}
At least one of $D_2, D_3$ is isolated component of $D$.
\end{lemma}

\begin{proof}
Assume that $D_2$ and $D_3$ meets irreducible components $D_4$ and $D_5$ of $D$ correspondingly. Note that $D_4^2=D_5^2=-2$. Let $\phi\colon\bar{X}\rightarrow\pp^1$ be a $\pp^1$-fibration defined by $|2\Gamma+D_4+D_5|$. Let $F$ be a fiber that contains $C, D_1$. Since $\Gamma$ meets every irreducible component of $D$ that meets $D_1$, we see that $F=C+D_1+E$, where $E$ is a $(-1)$-curve. Put $D_6$ is a $(-3)$-curve. Note that $\Gamma\cdot D_6=1$. Then $E\cdot D_6=2$. Let $g\colon\bar{X}\rightarrow Y$ be the contraction of $C$, let $Y\rightarrow\bar{Y}$ be the contraction of all $(-n)$-curves ($n=2,3$). We obtain a del Pezzo surface of degree one and with Picard number is equal one. We see that $g(E)$ is a minimal curve. Moreover, $|g(E)+\tilde{D}+K_Y|\neq\emptyset$, where $\tilde{D}$ is exceptional divisor of minimal resolution. Then, by Lemma \ref{Zhan1}, we see that $D_6$ is isolated component of $D$. Moreover, $E$ meets only two components of $D$, $D_1$ and $D_6$. Hence, $D_4$ and $D_5$ meets only $D_2, D_3$ and $D_1$ is an isolated component of $D$, i.e. $D^{(1)}=D_1$, $D^{(2)}=D_2+D_4$, $D^{(3)}=D_3+D_5$, $D^{(4)}=D_6$, a contradiction with $\#D=10$.
\end{proof}

So, we may assume that $D_1$ meets one irreducible component $D_4$ of $D$, $D_2$ meets one irreducible component $D_5$ of $D$, $D_3$ is an isolated component. Let $D_6$ be a $(-3)$-curve in $D-D_2-D_3$.

\begin{lemma}
\label{Case233-10}
The unique $(-3)$-curve in $D-D_2-D_3$ is a component of $D^{(4)}$.
\end{lemma}

\begin{proof}
Assume that $D_6$ is a component of $D^{(1)}$. Let $g\colon\bar{X}\rightarrow Y$ be the contraction of $C$, let $Y\rightarrow\bar{Y}$ be the contraction of all $(-n)$-curves where $n=2$ or $3$. Then we obtain a del Pezzo surface of Picard number one. Since $\Gamma$ intersects two irreducible components of $g(D^{(1)}$, we see that $\Gamma$ is a minimal curve, and  $|\Gamma+\tilde{D}+K_Y|\neq\emptyset$, where $\tilde{D}$ is exceptional divisor of minimal resolution. On the other hand, $\Gamma$ intersects two connected components of $\tilde{D}$, a contradiction.
\end{proof}

Now, by Theorem \ref{BMY}, we see that $D^(1) = D_1+D_4$, $D^(2) = D_2+D_5$, $D^{3} = D_3$, and $D^(4) = D_6$. Let $\phi\colon\bar{X}\rightarrow\pp^1$ be a $\pp^1$-fibration defined by $|2\Gamma + D_4+D_5|$. Then $D_6$ is a $2$-section, $D_1$ and $D_2$ are sections, and the remaining componenets of $D$ are fiber components. Consider the fiber $F$ containing $C$ and $D_3$. We see that $F = C+D_3+E_1+E_2$ where $E_1$ and $E_2$ are $(-1)$-curves, both intersecting $D_6$ and $D_3$. Let $F_2$ be the fiber containing the remaining components of $D$. By Lemma \ref{Zhan0} (1), $F_2$ has only one $(-1)$-curve, which is impossible.

Consider the case $(2,2,n)$. Assume that $2C+D_1+D_2+D_3+K_{\bar{X}}\sim 0$. Note that $D_1, D_2, D_3$ are isolated components of $D$. Let $\phi\colon\bar{X}\rightarrow\pp^1$ be a $\pp^1$-fibration defined by $|2C+D_1+D_2|$. We see that every singular fiber of $\phi$ contains only $(-1)$- and $(-2)$-curves. Moreover, by Lemma \ref{Zhan0} every singular fiber has only one $(-1)$-curve. Since $X$ has four singular points, we see that $\phi$ has one singular fiber of type (a) and one singular fiber of type (b). Then we have a double cover $D_3\rightarrow\pp^1$ with at least three ramification points, a contradiction with Hurwitz formula
(see Theorem \ref{Hurwitz}).

Assume that $2C+D_1+D_2+D_3+K_{\bar{X}}\sim \Gamma$. Let $\phi\colon\bar{X}\rightarrow\pp^1$ be a $\pp^1$-fibration defined by $|2C+D_1+D_2|$. We see that every singular fiber of $\phi$ contains only $(-1)$- and $(-2)$-curves. Moreover, every $(-1)$-curve in fibers is minimal. Let $F$ be a fiber of $\phi$ that contains $\Gamma$. We see that $\Gamma$ is a minimal curve. We may assume that $|\Gamma+D+K_{\bar{X}}|=\emptyset$. Then $\Gamma\cdot D^{(i)}\leq 1$, $i=1,2,3,4$. So, every $D_i$ ($i=1,2,3$) does not meet a $(-m)$-curve for some $m\geq 3$. Moreover, every $D_i$ ($i=1,2,3$) meets at most one component of $D$. Since $F$ consists of $(-1)$- and $(-2)$-curves, we see that $\Gamma$ does not meet $(-n)$-curve ($n\geq 3$). Then $D-D_3$ consists of $(-2)$-curves. Since $D_3$ meets at most one irreducible component $D_4$ of $D$, we see that the coefficient before $D_4$ in $D^{\sharp}$ is less than before $D_3$. Hence, $-\Gamma\cdot(K_{\bar{X}}+D^{\sharp})>-C\cdot(K_{\bar{X}}+D^{\sharp})$, a contradiction.

Assume that $C$ meets only one component $D_1$ of $D$. Let $\phi\colon\bar{X}\rightarrow Y$ be the consequence of contractions of $(-1)$-curves in $C+D$. We have two cases $\phi(C+D)$ consists of $(-n)$-curves ($n\geq 2$) and  $\phi(C+D)=\tilde{C}+\tilde{D}$, where $\tilde{C}$ is a $(-1)$-curve and $\tilde{D}$ consists of $(-n)$-curves. Moreover, $\tilde{C}$ meets at least two component of $\tilde{D}$. Assume that $\phi(C+D)$ consists of $(-n)$-curves ($n\geq 2$). Then the irreducible components of $\phi(D)$ generate the Picard group. On the other hand, $\phi(D)$ is negative definite, a contradiction with \ref{Hodge}. Assume that $\phi(C+D)=\tilde{C}+\tilde{D}$. Let $Y\rightarrow\bar{Y}$ be the contraction of all $(-n)$-curves ($n\geq 2$). Note that $\bar{Y}$ is a del Pezzo surface with log terminal singularities. Moreover, $\bar{Y}$ has at least five log terminal singularities, a contradiction.

So, we may assume that every minimal curve meets exactly two irreducible components of $D$.

Assume that $C$ meets two component $D_1$ and $D_2$ of $D$. We may assume that $D_1^2=-2$. Assume that $D_1, D_2$ correspond to singularities $P_1,P_2$. Put $P_3, P_4$ are another two singular points. Let $D^{(1)}, D^{(2)}, D^{(3)}, D^{(4)}$ be the connected component of $D$ correspond to $P_1, P_2, P_3, P_4$ correspondingly. Let $m_1, m_2, m_3, m_4$ be the orders of local fundamental groups of $P_1, P_2, P_3, P_4$ correspondingly.

Assume that $D_1$ meets two component of $D$. Assume that $D_2^2\leq -3$. Let $\bar{X}\rightarrow Y$ be the contraction of $C$ and $Y\rightarrow\bar{Y}$ be the contraction of all $(-n)$-curves ($n\geq 2$). Then $\rho(\bar{Y})=1$ and $\bar{Y}$ has five log terminal singularities, a contradiction. Assume that $D_2^2=-2$. Let $\phi\colon\bar{X}\rightarrow\pp^1$ be the $\pp^1$-fibration defined by $|2C+D_1+D_2|$. We see that every singular fibers of $\phi$ consists of $(-1)$- and $(-2)$-curves (see Lemma \ref{Pr1}). Let $F_1$ and $F_2$ be the singular fibers of $\phi$ that contains $D^{(3)}$ and $D^{(4)}$ correspondingly. Since every horizontal component of $D$ is a section, we see that $F_1$ and $F_2$ are of type (c). By Lemma \ref{Zhan0} we see that there exists at least three horizontal component of $D$. Then at least one $(-1)$-curve in $F_i$ ($i=1,2$) meets three component of $D$. Since this curve is minimal, we have a contradiction.

Assume that $P_1$ is a rational double point of type $A$.
Then there exists a $\pp^1$-fibration $\phi\colon\bar{X}\rightarrow\pp^1$ such that $C$ is contained in fiber of type (b). Moreover, $D_2$ is a 2-section. So, every singular fiber contains only $(-1)$- and $(-2)$-curves (see Lemma \ref{Pr1}). Assume that there exists a fiber of type (a) $2E+D_3+D_4$, where $E$ is a minimal curve and $D_3 D_4$ are components of $D$. If $E$ meets $D_2$, then there exists a minimal curve that meets three components of $D$. If $E$ does not meet $D_2$, then $D_3$ and $D_4$ meet $D_2$. So, $|E+D+K_{\bar{X}}|\neq \emptyset$, a contradiction. So, there exists no a fiber of type (a). Let $F_1$ and $F_2$ be the singular fibers of $\phi$ that contains $D^{(3)}$ and $D^{(4)}$ correspondingly. Since there exists no a singular fibers of type (a), we see that $F_1\neq F_2$. Assume that both $F_1$ and $F_2$ are of type (b). Then we have a double cover $D_3\rightarrow\pp^1$ with at least three ramification points, a contradiction with Hurwitz formula (see Theorem \ref{Hurwitz}). So, one of them is of type (c). We may assume that $F_1$ is of type (c). Then by Lemma \ref{Zhan0} there exists a section $D_3$ of $\phi$ that is a components of $D^{(1)}$. Hence, there exists a $(-1)$-curve in $F_1$ that meets three components of $D$. Since every $(-1)$-curve in $F_1$ is minimal, we have a contradiction.

So, we may assume that $P_1$ and $P_2$ are cyclic singularities. Moreover, $D_1$ meets only one irreducible component of $D$. Let $Y\rightarrow\bar{X}$ be the consequence of blow-ups of intersection point of $C$ and $D_2$, let $Y\rightarrow\bar{Y}$ be the contraction of all $(-n)$-curves ($n\geq 2$). We obtain a rational surface with Picard number one and with log terminal singularities. By Theorem \ref{BMY} we see that $P_3$ and $P_4$ are rational double points of type $A_1$.

\begin{lemma}
\label{CaseDD-1}
There exists no  $\pp^1$-fibration $\phi\colon\bar{X}\rightarrow\pp^1$ such that $\phi$ has no $n$-section  in $D$ with $n\geq 2$, every section of $\phi$ in $D$ is a component of $D^{(1)}+D^{(2)}$ and $\phi$ has at most three sections in $D$.
\end{lemma}

\begin{proof}
Put $D^{(3)}=D_3$. Note that $D_3$ is a $(-2)$-curves. Let $F$ be a singular fiber that contains $D_3$.

We claim that $F$ is of type (c). Indeed, since $D_3$ is an isolated component of $D$, we see that $D_3$ meets at most two $(-1)$-curves in $F$. Moreover, if $D_3$ meets two $(-1)$-curves in $F$, then $F$ is of type (c). Assume that $D_3$ meets only one $(-1)$-curve $E$ in $F$. Then multiplicity of $E$ in $F$ is at least two. Hence, $E$ does not meet sections in $D$. So, $E$ meets exactly two components $D_3$ and $D'_3$ of $D$. Let $W\rightarrow\bar{X}$ be blowups the intersection point of $E$ and $D'_3$ $k$ times. Let $W\rightarrow\bar{W}$ be the contraction of all $(-n)$-curves ($n\geq 2$). We see that $\rho(\bar{W})=1$ and $\bar{W}$ has only log terminal singularities. For sufficiently big $k$ we have a contradiction with \ref{BMY}.

Since $F$ is of type (c),  we can write  $F=D_3+E_1+E_2$ for some $(-1)$-curve $E_1$ and $E_2$. Since $D_3$ is isolated component of $D$, we see that $E_1, E_2$ meet every section of $\phi$. Note that $E_1$ and $E_2$ meet at least two components of $D$. Indeed, assume that $E_1$ meets only one component $D_3$ of $D$. Let $g\colon\bar{X}\rightarrow Y$ be the contraction $E_1+D_3$. We see that the components of $g(D)$ generate Picard group and $g(D)$ is negative definite, a contradiction with Hodge index theorem \ref{Hodge}. Since $\phi$ has at most three section in $D$, we see that at least one of $E_1, E_2$ meet only two components $D_3, D_4$ of $D$. We may assume that $E_1$ meets $D_3$ and $D_4$. Let $W\rightarrow\bar{X}$ be blowups the intersection point of $C$ and $D_2$ $k_1$ times and the intersection point of $E_1$ and $D_4$ $k_2$ times. $W\rightarrow\bar{W}$ be the contraction of all $(-n)$-curves ($n\geq 2$). We see that $\rho(\bar{W})=1$ and $\bar{W}$ has only log terminal singularities. For sufficiently big $k_1$ and $k_2$ we have a contradiction with \ref{BMY}.
\end{proof}

Assume that $D_1$ meets one component of $D$, $D_2$ meets two components of $D$. Let $g\colon\bar{X}\rightarrow Y$ be the consequence of contraction of $(-1)$-curves in $C+D^{(1)}$.

Assume that $D^{(1)}$ contains a $(-n)$-curves. So, $g(D^{(1)})$ is a linear chain of $(-n)$-curves. Let $Y\rightarrow\bar{Y}$ be the contraction of $g(D^{(1)})$, $g(D^{(2)}-D_2)$, $g(D^{(3)})$, $g(D^{(4)})$. Let $n=\#D$. Assume that $g$ contracts $m$ curves in $D$ and one $(-1)$-curve $C$. So, $\rho(Y)=n+1-(m+1)=n-m$, $\#g(D)=n-m$. In $Y\rightarrow\hat{Y}$ we contract $g(D)-g(D_2)$, i.e. we contract $n-m-1$ curves. So, $\rho(\hat{Y})=1$. We obtain a del Pezzo surface with Picard number one and with five singular points, a contradiction.
So, $D^{(1)}$ consists of $(-2)$-curves. Consider $g(D_2)$. Assume that $g(D_2)^2\leq -2$. Then the irreducible components of $g(D)$ generate Picard group and $g(D)$ is negative definite, a contradiction with Hodge index theorem \ref{Hodge}. Assume that $g(D_2)^2\geq 0$. Then there exists a $\pp^1$-fibration $\phi\colon\bar{X}\rightarrow\pp^1$ such that $\phi$ has two or three section in $D$ has no $n$-section in $D$ for $n\geq 1$. Moreover, all section of $\phi$ are components of $D^{(1)}$ or $D^{(2)}$, a contradiction (see Lemma \ref{CaseDD-1}). So, we may assume that $g(D_2)^2=-1$. Let $h\colon Y\rightarrow Y'$ be the consequence of contraction of $(-1)$-curves in $g(D^{(2)}$. We have two cases
\begin{enumerate}
\item $h(g(D^{(2)}))$ is a linear chain of $(-n)$-curves ($n\geq 2$). Then the irreducible components of $g(D)$ generate Picard group and $g(D)$ is negative definite. We have a contradiction with Hodge index theorem \ref{Hodge}.
\item There exists a $\pp^1$-fibration $\phi\colon\bar{X}\rightarrow\pp^1$ such that $\phi$ has one or two section in $D$ has no $n$-section in $D$ for $n\geq 1$. Moreover, all section of $\phi$ are components of $D^{(1)}$ or $D^{(2)}$ and there exist at most three sections of $\phi$, a contradiction (see Lemma \ref{CaseDD-1}).
\end{enumerate}

Assume that $D_1$ and $D_2$ meet at most one component of $D$. Let $g\colon\bar{X}\rightarrow Y$ be the consequence of contraction of $(-1)$-curves in $C+D^{(1)}+D^{(2)}$. As above, we have two cases
\begin{enumerate}
\item $g(D^{(2)})$ is a linear chain of $(-n)$-curves ($n\geq 2$). Then the irreducible components of $g(D)$ generate Picard group and $g(D)$ is negative definite, a contradiction to Hoge index theorem (Theorem  \ref{Hodge}).
\item There exists a $\pp^1$-fibration $\phi\colon\bar{X}\rightarrow\pp^1$ such that $\phi$ has one or two section in $D$ has no $n$-section in $D$ for $n\geq 1$. Moreover, all section of $\phi$ are components of $D^{(1)}$ or $D^{(2)}$ and there exist at most three sections of $\phi$, a contradiction (see Lemma \ref{CaseDD-1}).
\end{enumerate}

\section{The Case with a non-cyclic singular point that is not a rational double point}\label{noncycDV}
Throughout this section $X$ is a del Pezzo surface with 4 log terminal singular points. Assume that $X$ has a non-cyclic singular point $P$ that is not a rational double point.
Let $f\colon\tilde{X}\rightarrow X$ be the blow up of central component $D_1$ of the exceptional divisor over $P$ i.e. we contract every $(-n)$-curves ($n\geq 2$) on $\bar{X}$ except the central component. Since $P$ is non-cyclic, we see that $D_1$ passes through three singular points.  Let $Q_1, Q_2, Q_3$ be three such singular points of $\tilde{X}$. Consider the minimal resolution of singularities of $\tilde{X}$. Let $D^{(1)}, D^{(2)}, D^{(3)}$ be the exceptional divisors that correspond to $Q_1, Q_2, Q_3$. Note that $\tilde{X}$ has six singular points. Since the Mori  cone has two extremal rays, we see that there exists another morphism $g\colon\tilde{X}\rightarrow Z$ where $g$ is the contraction of an extremal ray. Note that we have either $Z=\pp^1$ and $g$ is a $\pp^1$-fibration; or $Z$ is a del Pezzo surface with $\rho(Z)=1$ (see Lemma \ref{WeekD}).

\subsection{$Z=\pp^1$} In this section we have $g\colon\tilde{X}\rightarrow \pp^1$.
In this case we can show that the curve $D_1$ is a horizontal component of $g$. More precisely, we have the following lemma.
\begin{lemma}
\label{DE1}
Assume that $P$ is not a rational double point. Then the curve $D_1$ is a section, a 2-section or a 3-section of $g$. Moreover, if $D_1$ is a 3-section of $g$, then the dual graph of minimal resolution of $P$ is the following
$$
\xymatrix@R=0.8em{
&&&&&\bullet\\
\star\ar@{-}[r]\ar@{-}[r]&\bullet\ar@{-}[r]&\bullet\ar@{-}[r]&\cdots\ar@{-}[r]&\bullet\ar@{-}[r]&\bullet\ar@{-}[r]\ar@{-}[u]&\bullet
}\eqno(1)
$$ where $\star$ denotes a $(-3)$-curve, $\bullet$ denotes a $(-2)$-curve.
\end{lemma}

\begin{proof}
Note that the coefficient $a_1$ of $D_1$ in $D^\sharp$ is at least $\frac{1}{2}$. Since a general fiber $C$ of $g$ is a $(0)$-curve and $C\cdot(K_{\bar{X}}+D^\sharp)<0$, we see that $C\cdot D_1\leq 3$. Moreover, if $a_1<\frac{2}{3}$, then $a_1=\frac{1}{2}$ and the dual graph of minimal resolution of $P$ is (1).
\end{proof}

\begin{lemma}
\label{DE2}
Assume that $D_1$ is a 3-section of $g$. Then $X$ has two rational double points of type $A_1$, one singular point $P$ which dual graph of minimal resolution is (1), and one singular point $Q$ which dual graph of minimal resolution is the following
$$
\xymatrix@R=0.8em{
&&\bullet\\
\bullet\ar@{-}[r]\ar@{-}[r]&\bullet\ar@{-}[r]&\ast\ar@{-}[r]\ar@{-}[u]&\star
}\eqno(2)
$$ where $\star$ denotes a $(-3)$-curve, $\bullet$ denotes a $(-2)$-curve, $\ast$ denotes a $(-(k-1))$-curve where $k$ is the number of irreducible components of the minimal resolution of $P$.
\end{lemma}

\begin{proof}
Let $\bar{g}\colon\bar{X}\rightarrow\pp^1$ be the $\pp^1$-fibration induced by $g$. Since $D_1$ is a unique component of $D$ that lies on $\tilde{X}$, we see that $D_1$ is a unique component of $D$ that is not in
any fiber of $\bar{g}$. By Lemma \ref{Zhan0} we see that every singular fiber of $\bar{g}$ contains only one $(-1)$-curve. Since every singular fiber of $g$ contains at most two singular points and multiplicity of $(-1)$-curve in fibers of $\bar{g}$ is at least two, we see that every singular fiber meets $D_1$ in one or two points. Let $P_1,P_2,P_3$ be the singular points of $\tilde{X}$ on $D_1$, and $Q_1,Q_2,Q_3$ be the remaining singular points of $\tilde{X}$ .


We may assume that $P_1$ and $P_2$ are singular points of type $A_1$ by Lemma \ref{DE1}. Let $F_1$ and $F_2$ be the fibers that contain $P_1$ and $P_2$. Since $P_1$ and $P_2$ are of type $A_1$, we see that $F_1$ and  $F_2$ do not pass through $P_3$; and $F_1\neq F_2$. Also, we see that $F_1$ contains a singular point $Q_1$, $F_2$ contains a singular point $Q_2$. Let $F_3$ be the fiber that contains $P_3$. Since $P_3$ is not a rational double point, we see that the multiplicity of $F_3$ is at least three. So, $F_3$ meets $D_1$ in one point $P_3$ and $F_3$ contains $Q_3$. So, by  Hurwitz formula, we see that $F_1$ and $F_2$ is of type (a) in Lemma \ref{Zhan0}.

So, we have the following picture of the dual graph $$
\xymatrix@R=0.8em{
&\bullet&&&&&&\bullet\ar@{-}[r]&\circ\ar@{-}[r]\ar@{-}[rd]&\bullet&\circ\ar@{-}[r]&\bullet\\
\star\ar@{-}[r]&\ast\ar@{-}[u]\ar@{-}[r]&\bullet\ar@{-}[r]&\bullet\ar@{-}[r]&\circ\ar@{-}[r]\ar@{-}[r]&\star\ar@{-}[r]&\bullet\ar@{-}[r]&\cdots\ar@{-}[r]&\bullet\ar@{-}[r]&\bullet\ar@{-}[r]\ar@{-}[u]\ar@{-}[ru]&\bullet\ar@{-}[u]
}
$$  where $\star$ denotes a $(-3)$-curve, $\bullet$ denotes a $(-2)$-curve, $\circ$ denotes a $(-1)$-curve, $\ast$ denotes a $(-(k-1))$-curve, $k$ is the number of irreducible components of the minimal resolution of $P$.
\end{proof}

\begin{remark}
Note that by replacing $D_1$ to component of $D$ that correspond to $\ast$ in Lemma \ref{DE2}, we may assume that $D_1$ is either a section or a $2$-section, hence reduced to  Section \ref{P1Fib}.
\end{remark}

Now Lemma \ref{DE3} and Lemma \ref{DE4} complete the proof.

\subsection{$Z$ is a surface.} Assume that $g$ is birational. Then $g$ contracts a curve $E$. By Lemma \ref{WeekD} $Z$ is a log del Pezzo surface of Picard number one. Since $X$ has $4$ singular points, $E$ passes through at least two singular points. Assume that $E$ passes through three singular points $Q_1, Q_2, Q_3$. Consider a minimal resolution of singularities of $\tilde{X}$. Let $D^{(1)}, D^{(2)}, D^{(3)}$ be the exceptional divisors over $Q_1, Q_2, Q_3$, respectively. Since $E+D^{(1)}+D^{(2)}+D^{(3)}$ is negative definite, the proper transform of $E$ is a $(-1)$-curve. By abusing notation, the proper transform of $E$ is also denoted by $E$ if there is no confusion. Let $D'_1,D'_2,D'_3$ be the irreducible components of $D$ intersecting $E$ where $D'_i$ is a component of $D^{(i)}$ for each $i=1,2,3.$ Since $E\cdot(K_{\bar{X}}+D^{\sharp})<0$, we see that at least one of $D'_1,D'_2,D'_3$ is a $(-2)$-curve. Say it is $D'_1$. Assume that $D'_2$ is also a $(-2)$-curve, then $(2E+D'_1+D'_2)^2=0$, a contradiction. Since $E\cdot(K_{\bar{X}}+D^{\sharp})<0$, we see that we have the followings cases for $(-D'^2_2,-D'^2_3)$: (3,3),(3,4),(3,5). On the other hand, $$(4E+2D'_1+D'_2+D'_3)^2=8+D'^2_2+D'^2_3\geq 0,$$
a contradiction. Thus $E$ passes through exactly two singular points. Since $\tilde{X}$ has six singular points, $Z$ has exactly four log terminal singular points by \cite{Bel1}.

Let $P, P_1, P_2, P_3$ be the singular points of $X$; and  $Q_1, Q_2, Q_3$ be the singular points of $\tilde{X}$ on $D_1$.
\subsubsection{Assume that $E$ passes through two singular points of $X$: one on $D_1$ and the other outside $D_1$.}
We may assume that $E$ passes through $Q_1$ and $P_1$. Consider the consequence of contractions of $(-1)$-curves in $E+D$ on the minimal resolution $\bar{X}$ of $\tilde{X}$. We obtain a sequence of blow-downs:  $$\bar{X}\rightarrow X_1\rightarrow X_2\rightarrow\cdots\rightarrow X_n$$ such that each $X_k$ is a del Pezzo surface with log terminal singularities. Let $g_k\colon\bar{X}\rightarrow X_k$ be the composition of such morphisms from $\bar{X}$ to $X_k$. We have two cases.

$(1)$ There exists $X_k$ such that $g_k(D)$ contains a $(-1)$-curve $E'$ that intersects three other components of $g_k(D)$. Then $g_k(D)-E'$ contains five connected component. Let $X_k\rightarrow Y$ be the contraction of $g_k(D)-E'$. We see that $Y$ is a del Pezzo surface with $\rho(Y)=1$ and five singular points. A contradiction.

$(2)$ There exists $X_k$ such that $g_k(D)$ contains a $(-1)$-curve $E'$ that intersects exactly one component $D'_1$ of $g_k(D)$. Let $W\rightarrow X_k$ be the blowups of intersection point of $E'$ and $D'_1$ $k$ times, let $W\rightarrow \bar{W}$ be the contraction of all $(-n)$-curves ($n\geq 2$). We see that $\rho(\bar{W})=1$ and $\bar{W}$ has only log terminal singularities. For sufficiently big $k$, by Theorem \ref{BMY}, we see that $P_2$ and $P_3$ are of type $A_1$ and $\bar{W}$ contains two singular points of type $A_1$, one singular point of type $A_k$ and other cyclic singular point. This cases is considered in Section \ref{NoEmpty} and Section \ref{Empty}, and such case does not exist, a contradiction.

\subsubsection{Assume that $E$ passes through two singular points of $X$ outside $D_1$.} We may assume that $E$ passes through $P_1$ and $P_2$.
\begin{lemma}
\label{DE5}
Under the above assumptions,  $X$ contains three singular points that are not rational double points and the dual graph of the minimal resolution of $P$ is the following
$$
\xymatrix@R=0.8em{
&&&&&\bullet\\
\star\ar@{-}[r]\ar@{-}[r]&\bullet\ar@{-}[r]&\bullet\ar@{-}[r]&\cdots\ar@{-}[r]&\bullet\ar@{-}[r]&\bullet\ar@{-}[r]\ar@{-}[u]&\bullet
}\eqno(1)
$$ where $\star$ denotes a $(-3)$-curve, $\bullet$ denotes a $(-2)$-curve.
\end{lemma}

\begin{proof}
Since $E$ contracts to a smooth point, we see that $E$ meets a $(-2)$-curve $D_2$ and a $(-n)$-curve $D_3$ with $n\geq 3$. Since $-E\cdot(K_{\bar{X}}+D^{\sharp})>0$, we see that the dual graph of the minimal resolution of $P$ is the following
$$
\xymatrix@R=0.8em{
&&&&&\bullet\\
\star\ar@{-}[r]\ar@{-}[r]&\bullet\ar@{-}[r]&\bullet\ar@{-}[r]&\cdots\ar@{-}[r]&\bullet\ar@{-}[r]&\bullet\ar@{-}[r]\ar@{-}[u]&\bullet
}\eqno(1)
$$ where $\star$ denotes a $(-3)$-curve, $\bullet$ denotes a $(-2)$-curve. Moreover, $D_3$ is a $(-3)$-curve. Let $C$ be a minimal curve. Since $P$ is the non-cyclic singular points that is not a rational double point, we see that $C$ meets different connected components of $D$. So, if $P_3$ is a rational double point, then $E$ is a minimal curve. Since $E$ meets three components $D_1,D_2,D_3$ of $D$ and $D_1$ meets three other components of $D$, we see that $2E+D_1+D_2+D_3+K_{\bar{X}}\sim\Gamma$, where $\Gamma$ is a $(-1)$-curve (see Lemma \ref{Zhan3}). Then $\Gamma\cdot(K_{\bar{X}}+D^{\sharp})\geq 0$, a contradiction.
\end{proof}

By Lemma \ref{DE5}, we see that $Z$ has two singular points $Q_1$ and $Q_2$ of type $A_1$, and a cyclic singular point  $Q_3$, and a singular point $P_3$ that is not a rational double point.

We claim that $P_3$ is a noncyclic singular point. Indeed, if otherwise, then $Z$ has four cyclic singular points: two of them are of type $A_1$ and the remaining two are not rational double points, a contradiction with Section \ref{NoEmpty} and Section \ref{Empty}. Let $D^{(1)},D^{(2)},D^{(3)},D^{(4)}$ be the connected component of $D$ over $P,P_1,P_2,P_3$,
correspondingly.

Let $C_Z$ be a minimal curve. By Theorem \ref{BMY} we see that $D^{(2)},D^{(3)}$ consist of one irreducible component. Since $Z$ contains a noncylic singular point which is not a rational double point, $C_Z$ meets each connected component of $D_Z$ at most once.

Assume $C$ meets three components of $D$. Assume that $C$ meets $D^{(2)},D^{(3)},D^{(4)}$. Put $D_4$ is a component of $D^{(4)}$ that meets $C$. Assume that $D^2_4=-2$. Consider a $\pp^1$-fibration $\phi\colon\bar{X}\rightarrow\pp^1$ defined by $|2C+D_2+D_3|$. By Lemma \ref{Pr1} we see that every singular fiber consists of $(-1)$- and $(-2)$-curves. On the other hand, there exists a fiber that contains $D^{(1)}$, a contradiction. So, $D_4^2=-3$ and $C\cdot(K_{\bar{X}}+D^{\sharp})=E\cdot(K_{\bar{X}}+D^{\sharp})$. Hence, $E$ is also a minimal curve. As above, we have a contradiction. Assume that $C$ meets $D^{(1)},D^{(2)},D^{(3)}$. Then $C\cdot(K_{\bar{X}}+D^{\sharp})=E\cdot(K_{\bar{X}}+D^{\sharp})$. A contradiction. So, $C$ meets $D^{(1)}$ and $D^{(4)}$. Put $D_4$ is a component of $D^{(4)}$ that meets $C$ and $D_5$ is a component of $D^{(1)}$ that meets $C$ (maybe $D_1=D_5$). Assume that $C$ meets $D^{(1)},D^{(2)},D^{(4)}$. Since $C\cdot(K_{\bar{X}}+D^{\sharp})<0$, we see that either $D_4$ either $D_5$ is a $(-2)$-curve. Consider a $\pp^1$-fibration $\phi\colon\bar{X}\rightarrow\pp^1$ defined by $|2C+D_2+D_i|$ ($i=4,5$). By Lemma \ref{Pr1} we see that every singular fiber consists of $(-1)$- and $(-2)$-curves. On the other hand, there exists a fiber that contains $D^{(3)}=D_3$, a contradiction. Assume that $C$ meets $D^{(1)},D^{(3)},D^{(4)}$. Since $D^{(1)}$ does not consist of one curve, we see that $2C+D_3+D_4+D_5+K_{\bar{X}}\sim\Gamma$, where $\Gamma$ is a $(-1)$-curve. Then $D_1\neq D_5$ and $E$ does not meet $C$. Assume that $D_4$ and $D_5$ are $(-2)$-curves. Consider a $\pp^1$-fibration $\phi\colon\bar{X}\rightarrow\pp^1$ defined by $|2C+D_4+D_5|$. By Lemma \ref{Pr1} we see that every $(-1)$-curve in fibers is minimal. On the other hand, there exists a fiber that contains $E$, a contradiction. Since $C\cdot(K_{\bar{X}}+D^{\sharp})<0$, we see that $D_4$ is a $(-2)$-curve. Note that $\Gamma$ meets the curve in $D^{(1)}$ that meets $D_5$ and a $(-n)$-curve ($n\geq 3$) in $D^{(4)}$. Then $\Gamma\cdot(K_{\bar{X}}+D^{\sharp})>0$, a contradiction.

Assume $C$ meets two components $D_i$ and $D_j$ of $D$. First, consider the case that both $D_i$ and $D_j$ are $(-2)$-curves. Consider a $\pp^1$-fibration $\phi\colon\bar{X}\rightarrow\pp^1$ defined by $|2C+D_i+D_j|$. By Lemma \ref{Pr1} we see that every singular fiber consists of $(-1)$- and $(-2)$-curves. On the other hand, there exists a fiber that contains $D^{(3)}=D_3$, a contradiction. Now, consider the case that $D_i$ is a $(-2)$-curve and $D_j$ is a $(-3)$-curve. Assume that $D_j$ is $D_2$. So, we can blowup intersection point of $C$ and $D_i$. We obtain a surface $Y$. Let $Y\rightarrow\bar{Y}$ be the contraction of all $(-n)$-curves ($n\geq 2$). Then $\bar{Y}$ has only log terminal singularities and $\rho(\bar{Y})=1$. A contradiction with \ref{BMY}. Assume that $D_j$ is $D_3$. Then $D_i$ is a component of $D^{(1)}$ or $D^{(4)}$ and $D_i^2=-2$. Let $h\colon\bar {X}\rightarrow Y$ be the contraction of $C$. Put $h'\colon W\rightarrow Y$ are blowups of intersection point of $h(D_i)$ and $h(D_j)$ $k$ times. Let $W\rightarrow\bar{W}$ be the contraction of all $(-n)$-curves ($n\geq 2$). Then $\bar{W}$ has only log terminal singularities and $\rho(\bar{W})=1$. For sufficiently big $k$ we have a contradiction with \ref{BMY}.

So, $C$ meets two components of $D^{(1)},D^{(4)}$. Assume that $D_i^2=-2$, $D_j^2=-n$ ($n\geq 3$). Since $C\cdot (K_{\bar{X}}+D^{\sharp})<0$, we see that $D_i$ meets $(-2)$-curve $D_k$.
Consider a $\pp^1$-fibration $\phi\colon\bar{X}\rightarrow\pp^1$ defined by $|3C+D_j+2D_i+D_k|$. Let $F$ be the fiber of $\phi$ that contains $D_3$. Since $D_3^2=-1$, we see that the sum of multiplicity of $(-1)$-curves in $F$ is at least three. So, every $(-1)$-curve in $F$ is minimal, a contradiction.

Assume that $C$ meets exactly one component $D_i$ of $D$. By considering the consequence of contractions of $(-1)$-curves in $C+D$, we arrive at the case where either $C+D$ is negative definite or the image of $D$ have at least five connected components, a contradiction.

\subsubsection{Assume that $E$ passes through two singular points on $D_1$.} We may assume that $E$ passes through $Q_1$ and $Q_2$. 
Let $h\colon\bar{X}\rightarrow X'$ be the contraction of $E$. By abusing notation, we also denote $D_1$  for the images of $D_1$. Then there exists a linear chain of negative rational curves with self-intersection number at least $-2$, and a $(-1)$-curve $E'$ that intersects the end component of the linear chain and $D_1$. Indeed, if $Q_1$ or $Q_2$ is of type $A_1$, then after first contraction we have this picture. If both $Q_1$ and $Q_2$ are not of type $A_1$, then $P$ is of type $E$, and we get the conclusion after the second contraction.

Moreover, consider the consequence of the contractions of $(-1)$-curves in $E'+D$. Since $E\cdot(K_{\bar{X}}+D^{\sharp})<0$, we see that on some step we again obtain a linear chain of negative rational curves with self-intersection number at least $-2$, and a $(-1)$-curve $E'$ that intersects the end component of the linear chain and one  component that is not an end component. Indeed, since $E$ contracts $Q_1$ and $Q_2$ to a smooth point, if $E'$ meets $D_1$ and $D_k$, then we see that $D_1$, $D_2$, \ldots, $D_{k+1}$ are $(-2)$-curves.

We claim that  at least one of $P_1,P_2,P_3$ is not a rational double point. Indeed, if otherwise, our surface $T$ is a del Pezzo surface of Picard number one with at worst rational double points. By the classification, we see that the type of singularities is $2A_1+2A_3$. But one can show that there is no $(-1)$-curve $E_T$ in the minimal resolution of $T$, intersecting the first two (or equivalently the last two) irreducible components of the exceptional divisor over a singular point of type $A_3$ by using \cite[Proposition 4.2]{Hw13} or \cite[Theorem 1]{Hw19}. This is a contradiction.

Assume that every $P_i$ is cyclic. Let $W\rightarrow\bar{X}$ be the blowup of intersection point of $E'$ and one middle component.
Let $W\rightarrow\bar{W}$ be the contraction of all $(-n)$-curves ($n\geq 2$). Then $\bar{W}$ is a del Pezzo surface log terminal singularities and $\rho(\bar{W})=1$. Moreover, $\bar{W}$ has four cyclic singular point and two of them are non rational double points. A contradiction. Put $h'\colon W\rightarrow X'$ are blowups of intersection point of $E'$ and $D'_1$ $k$ times. Let $W\rightarrow\bar{W}$ be the contraction of all $(-n)$-curves ($n\geq 2$). Then $\bar{W}$ has only log terminal singularities and $\rho(\bar{W})=1$. So, by Theorem \ref{BMY} we see that $P_1$ and $P_2$ are of type $A_1$ and $P_3$ is a non-cyclic singular point that is not a rational double point.

Let $\psi\colon\tilde{X}'\rightarrow X$ be the blowup of one middle  component of $P_3$. Let $g'\colon\tilde{X}'\rightarrow Y'$ be the other contraction. Since we consider every cases except when $g'$ contract two singular point that lie on exception curve of $\psi$. As above, there exists birational transformation $\bar{X}$ to $W$, where $W$ contains two isolate $(-2)$-curves and two linear chains of $(-2)$-curves. Moreover, there exist two $(-1)$-curves $E_1,E_2$ that meet ends of each linear chain and central components. By the classification, we see that the type of singularities is $2A_1+2A_3$, a contradiction as above.

This completes the proof of   Theorem \ref{MainT}.

\end{document}